\documentclass[times,10pt]{article}

\usepackage{microtype}
\usepackage{graphicx}
\usepackage{xcolor}
\usepackage{subfigure}
\usepackage{booktabs} % for professional tables
\usepackage{multirow}
\usepackage{makecell}
\makeatletter
\def\Cline#1#2{\@Cline#1#2\@nil}
\def\@Cline#1-#2#3\@nil{%
  \omit
  \@multicnt#1%
  \advance\@multispan\m@ne
  \ifnum\@multicnt=\@ne\@firstofone{&\omit}\fi
  \@multicnt#2%
  \advance\@multicnt-#1%
  \advance\@multispan\@ne
  \leaders\hrule\@height#3\hfill
  \cr}
\makeatother

\usepackage[accepted]{MyStyle}

\usepackage{amsmath}
\usepackage{amssymb}
\usepackage{mathtools}
\usepackage{amsthm}
\usepackage{bbm}

\usepackage{hyperref}

\usepackage[capitalize,noabbrev]{cleveref}

\theoremstyle{plain}
\newtheorem{thm}{Theorem}[section]
\newtheorem{prop}[thm]{Proposition}
\newtheorem{lem}[thm]{Lemma}
\newtheorem{cor}[thm]{Corollary}
\theoremstyle{definition}
\newtheorem{defn}[thm]{Definition}

\newtheorem{ex}[thm]{Example}
\theoremstyle{remark}
\newtheorem{rem}[thm]{Remark}

\usepackage[textsize=tiny]{todonotes}

\newcommand{\one}{\mathbbm{1}}
\newcommand{\braces}[1]{{\rm (}#1{\rm )}}

\newcommand{\<}{\langle}
\renewcommand{\>}{\rangle}

\newcommand{\wt}{\widetilde}
\newcommand{\wh}{\widehat}
\newcommand{\dist}{\operatorname{dist}}

\newcommand{\R}{\ensuremath{\mathbb R}}    % Real numbers
\newcommand{\C}{\ensuremath{\mathbb C}}    % Complex numbers
\newcommand{\Q}{\ensuremath{\mathbb Q}}    % Rational numbers
\newcommand{\N}{\ensuremath{\mathbb N}}    % Natural numbers
\newcommand{\Z}{\ensuremath{\mathbb Z}}    % Integers
\newcommand{\T}{\ensuremath{\mathbb T}}    % Unit circle

\newcommand{\calB}{\mathcal B}         \newcommand{\frakB}{\mathfrak B}
         
\newcommand{\calD}{\mathcal D}         
         
\newcommand{\calF}{\mathcal F}         
         
\newcommand{\calH}{\mathcal H}

\newcommand{\calL}{\mathcal L}

\newcommand{\calX}{\mathcal X}

         \newcommand{\bE}{\mathbb E}

         \newcommand{\bP}{\mathbb P}

         \newcommand{\bV}{\mathbb V}

\newcommand{\la}{\lambda}
\newcommand{\veps}{\varepsilon}

\newcommand{\vphi}{\varphi}

\renewcommand{\Re}{\operatorname{Re}}
\newcommand{\linspan}{\operatorname{span}}

\newcommand{\Lra}{\Longrightarrow}

\newcommand{\upto}{\uparrow}

\begin{document}

\onecolumn
\icmltitle{Variance representations and convergence rates for data-driven approximations of Koopman operators}
%Sharp bounds on the sampling efficiency of Extended Dynamic Mode Decomposition} 
%Sharp Finite-Data Error Bounds for highly sample-efficient Extended Dynamic Mode Decomposition}




\begin{icmlauthorlist}
\icmlauthor{Friedrich M. Philipp}{yyy}
\icmlauthor{Manuel Schaller}{yyy}
\icmlauthor{Septimus Boshoff}{zzz}
\icmlauthor{Sebastian Peitz}{zzz}
\icmlauthor{Feliks N\"uske}{aaa,bbb}
\icmlauthor{Karl Worthmann}{yyy}
%\icmlauthor{Firstname7 Lastname7}{comp}
%\icmlauthor{Firstname8 Lastname8}{sch}
%\icmlauthor{Firstname8 Lastname8}{yyy,comp}
\end{icmlauthorlist}

\icmlaffiliation{yyy}{Technische Universit\"at Ilmenau, Optimization-based Control Group, Institute of Mathematics, Ilmenau, Germany}
\icmlaffiliation{zzz}{Paderborn University, Department of Computer Science, Paderborn, Germany}
\icmlaffiliation{aaa}{Max Planck Institute for Dynamics of Complex Technical Systems, Magdeburg, Germany}
\icmlaffiliation{bbb}{Freie Universität Berlin, Department of Mathematics and Computer Science}

\icmlcorrespondingauthor{Friedrich Philipp}{friedrich.philipp@tu-ilmenau.de}
%\icmlcorrespondingauthor{Firstname2 Lastname2}{first2.last2@www.uk}

\icmlkeywords{Koopman operator, Machine Learning, Extended Dynamic Mode Decomposition, Error Bounds, Ergodic Sampling}

\vskip 0.3in



\printAffiliationsAndNotice % otherwise use the standard text.

%\onecolumn
\begin{abstract} %4-6 sentences allowed
We rigorously derive novel error bounds for extended dynamic mode decomposition (EDMD) to approximate the Koopman operator for discrete- and continuous time (stochastic) systems; both for i.i.d.\ and ergodic sampling under non-restrictive assumptions. We show exponential convergence rates for i.i.d.\ sampling and provide the first superlinear convergence rates for ergodic sampling of deterministic systems. The proofs are based on novel exact variance representations for the empirical estimators of mass and stiffness matrix. Moreover, we verify the accuracy of the derived error bounds and convergence rates by means of numerical simulations for highly-complex dynamical systems including a nonlinear partial differential equation.
\end{abstract}

\section{Introduction}
Extended Dynamic Mode Decomposition (EDMD; \cite{WillKevr15}) is one of the most commonly used machine-learning methods for identifying highly-nonlinear and, in addition, possibly infinite-dimensional dynamical systems from data. At its heart, EDMD provides a data-driven approach to learn the Koopman operator~\cite{Koop31} propagating a finite number of observable functions along the flow, which results in a purely data-driven and well-interpretable surrogate model for analysis, prediction, and control, e.g., based on identified symmetries for data augmentation~\cite{WeisSinh22} or deep learning~\cite{HanEule21}. Since the Koopman operator is linear, %operator, 
it serves as a %this connection enables 
powerful tool to leverage %and 
well-established concepts %tools 
from approximation, ergodic, operator, and %theory, 
statistical-learning theory in %utilize 
certifiable %and highly interpretable 
machine learning also in safety-critical applications. %guarantee the certifiability of this machine learning algorithm. % the data-driven EDMD method.

Initiated in \cite{Mezi04,Mezi05}, EDMD and Koopman-based methods have successfully enabled data-driven simulations and analysis of various highly complex applications, such as molecular dynamics~\cite{SchuetteKoltaiKlus2016,WuNusk17,KlusNueske2018}, nonlinear partial differential equations including turbulent flows \cite{GianKolc18,Mezi13}, quantum mechanics~\cite{klus2022koopman}, neuroscience~\cite{BrunJohn16}, deep learning \cite{DR20}, electrocardiography~\cite{GolaRady21} or climate prediction~\cite{AzenEri20} to name just a few. %\textcolor{red}{(SP: Hab mal eine Referenz für DL-Anwendungen hinzugefügt. Dachte das könnte für die Community interessant sein.)}
For further applications and Koopman-related learning architectures, we refer to \citet{KutzBrun16,Maur20,BrunBudi22,RetcAmos23}.

%For $C^\infty$-dynamics with the associated Koopman operator having purely discrete and non-dense spectrum on $\T$, it was proved in \cite{Mezi22} that $\|C^{-1}C_+ - \wh C^{-1}\wh C_+\|_F\lesssim m^{-1}$ almost surely. 

Convergence of EDMD to the Koopman operator in the infinite data limit was proven in \citet{KordaMezi18}. However, despite the enormous success of EDMD, error bounds depending on the number of %for finitely many 
data samples are still scarce. The first finite-data error bounds were given in~\citet{Mezi22}: for deterministic systems based on ergodic sampling under the rather strong assumption that the spectrum of the Koopman operator is discrete and non-dense on the unit circle. First results for i.i.d.\ (independently and identically distributed) sampling of systems governed by ordinary differential equations (ODEs) can be found in~\citet{ZhanZuaz23}. For stochastic systems, finite-data error bounds were firstly derived %provided 
in~\citet{NuskPeit23} under both i.i.d.\ and ergodic sampling; including an extension to control systems. However, the result for ergodic sampling hinges on the exponential stability of the Koopman semigroup. %--- an assumption %, which %significantly 
%limiting their applicability. %range of applications. 
For observable functions contained in a Reproducing Kernel Hilbert Space (RKHS), bounds were only recently provided for prediction \cite{PhilScha23} and control \cite{PhilScha23b} of continuous-time systems, by \cite{KostLoun23} and \cite{KostNov22} for i.i.d.\ sampling from an invariant measure, and by \cite{LiMeun22} and \cite{CiliRosa20} for conditional mean embeddings which are closely related to the Koopman operator. For a treatment of the $L^\infty$-estimation error of Koopman operator estimators for deterministic systems see \cite{KoehPhil24}. However, although EDMD can be viewed as a kernel regression method on a finite-dimensional RKHS (the space spanned by the observables), the above results do in general not apply, since in this special finite-dimensional case, the imposed assumptions reduce to the so-called {\em well-specified setting}, where the RKHS is invariant under the Koopman operator, which cannot be expected for finite-dimensional function spaces. In conclusion, so far only finite-data error bounds exist for continuous-time systems and, if ergodic sampling is considered, under rather %quite 
restrictive conditions. Moreover, all bounds decay at most linearly in the amount of data used for estimation.

%First finite-data error bounds for a class of finite-dimensional systems were given in \citet{ZhanZuaz23} for systems governed by nonlinear ordinary differential equations (ODEs) based on i.i.d.~(independently and identically distributed) sampling and in~\citet{NuskPeit23} for nonlinear stochastic differential equations (SDEs) based on i.i.d.\ and ergodic sampling, including an extension to control systems. 
%However, for ergodic sampling, the restrictive assumption of exponential stability was imposed on the Koopman semigroup, which significantly limited the range of applications. 
%For observable functions contained in a Reproducing Kernel Hilbert Space (RKHS), bounds were only recently provided for prediction \cite{PhilScha23} and control \cite{PhilScha23b} of continuous-time systems, and by \cite{KostLoun23} and \cite{KostNov22} for i.i.d.\ sampling from an invariant measure. In conclusion, so far only finite-data error bounds exist for continuous-time systems and, if ergodic sampling is considered, under quite restrictive conditions. Moreover, all bounds decay at most linearly in the amount of data used for estimation.

In this work, we focus on the estimation error for EDMD and provide a complete analysis of the 
%estimation error of EDMD 
latter for a much larger %very general 
class of Markov processes in Polish spaces. 
In particular, we remove the (restrictive) requirement of exponential stability and the restriction to systems governed by stochastic differential equations in~\cite{NuskPeit23}. This broadens the class of systems covered by our results such that, in addition, %was assumed the restrictive and \red{potentially hard-to-verify}
%hardly verifiable
%assumptions like exponential stability. \red{[FN: let's tone this down a little, there are many cases where exponential stability can be rigorously proven]}
%This includes stochastic differential equations, 
nonlinear partial differential equations and discrete-time Markov processes are included. In particular, all systems considered in~\cite{RozwMehr23} are covered rendering the proposed techniques accessible for a profound error analysis w.r.t.\ the number of used data samples. %ordinary differential equations. 
We provide 
%\textit{sharp} 
bounds on the convergence rate and, thus, the sampling efficiency of EDMD, i.e., Koopman-based machine learning for both i.i.d.~data and ergodic sampling.
Since data can be collected from a single (sufficiently-long) trajectory, which considerably facilitates the data collection process, ergodic sampling is of particular interest for many  practical applications as demonstrated in our examples. 

%The main contributions can be summarized as follows:%\\[-9mm]

Our contribution in this work is three-fold:
\begin{enumerate}
    \item [(1)] We severely weaken assumptions made in previous works for ergodic sampling of stochastic systems and prove
    %sharp 
    error bounds with a linear rate. %convergence. %\\[-7mm]
    \item [(2)] We derive the first error bounds showing %with %. In particular, we present the first bounds with a 
    superlinear convergence % rate. 
    for ergodic sampling.
    \item [(3)] We establish all 
    our results for continuous- and discrete-time systems --~for %; both %We extend both the existing results as well as our newly derived finite-data error bounds also to discrete-time deterministic and stochastic Markov processes with 
    i.i.d.\ and ergodic sampling.
\end{enumerate}

% Main contributions:
% \begin{itemize}
% \item New bounds on EDMD prediction error with explicit constants for both i.i.d.\ and ergodic sampling.
% \item i.i.d.\ sampling: As opposed to other works (e.g., \cite{Moll21}), we allow the dictionary functions to be unbounded (e.g., polynomials on $\R^d$).
% \item ergodic sampling: Explicit formulas for variance of data matrices.
% \item ergodic sampling: Weak assumptions for error analysis ($\la=1$ isolated simple eigenvalue of Koopman operator instead of exponential stability as in \cite{NuskPeit23}). No mixing assumed.
% \item ergodic sampling: We address the deterministic case with a thorough ergodicity analysis and obtain better bounds.
% \end{itemize}

\noindent \textbf{Notation}: We denote the constant function $x\mapsto 1$ by $\one$. Furthermore, the notation $\<\cdot,\cdot\>_F$ and $\|\cdot\|_F$ is used for the Frobenius scalar product on~$\R^{n\times m}$ and its corresponding norm, respectively. We use the notation $[n:m] := \mathbb{Z} \cap [n,m]$. For a probability measure $\mu$, the scalar product and the norm on~$L^2(\mu)$ are denoted by $\<\cdot,\cdot\>$ and $\|\cdot\|$, respectively. The orthogonal projection w.r.t.\ a closed subspace $M$ in $L^2(\mu)$ is denoted by $P_M$.

\section{EDMD: Data-driven prediction of nonlinear dynamics}

Let $(X_n)_{n\in\N_0}$ be a time-homogeneous discrete-time Markov process taking its values in a Polish space~$\calX$.
%$X_n$ may also arise from 
One representative system class is given by \textit{discrete-time} dynamical systems
\begin{equation}\tag{DTDS}\label{eq:DTDS}
    x_{n+1} = T(x_n) + \veps_n
\end{equation}
with independently and identically distributed (i.i.d.) noise~$\veps_n$.
Another system class contained in our setting are processes, which arise from samples of a time-homogeneous {\em continuous-time} Markov process $(Y_t)_{t\ge 0}$, i.e., $X_n = Y_{n\Delta t}$ with sampling period~$\Delta t > 0$. 
The process~$Y_t$ might, e.g., be the solution of a stochastic differential equation
\begin{equation}\tag{SDE}\label{eq:SDE}
    dY_t = f(Y_t)\,dt + \sigma(Y_t)\,dW_t.
\end{equation}
%The same holds for 
This includes their deterministic analogues, i.e., \eqref{eq:DTDS} with $\veps_n \equiv 0$ and \eqref{eq:SDE} with $\sigma \equiv 0$.

Let $\rho : \calX\times\frakB(\calX)\to [0,1]$ denote the transition kernel associated with $(X_n)$, where $\frakB(\calX)$ denotes the Borel sigma algebra on $\calX$, i.e., $\rho(x,A) = \bP(X_{n+1} \in A \mid X_n = x)$. %for all $n \in \mathbb{N}_0$. 
Then we have
%, we may also write 
$\bP^{X_{n+1}}(A) = \int\rho(x,A)\,d\bP^{X_n}(x)$, where $\bP^X$ denotes the law of a random variable $X$. %Recall that $\rho$ connects the laws of $X_{n+1}$ and $X_n$ in the following way:
Using the notation $\rho_0(x,A) := \delta_x(A)$ with the Dirac measure~$\delta_x$, we iteratively define $\rho_{n+1}(x,A) := \int\rho_n(y,A)\,\rho(x,dy)$ for $x\in\calX$ and $A\in\frakB(\calX)$, where $\rho(x,dy)$ stands for $d\rho(x,\cdot)(y)$. 
Then, $\rho = \rho_1$ and $\bP^{X_{n}}(A) = \int\rho_n(x,A)\,d\bP^{X_0}(x)$ hold.
For discrete-time deterministic dynamics, i.e., \eqref{eq:DTDS} with $\varepsilon_n \equiv 0$, we have $\rho(x,A) = \delta_x(T^{-1}(A))$.

%\subsection
\paragraph{The Koopman operator and EDMD.} 
%
%For $p \in [1,\infty)$ and any probability measure~$\mu$ satisfying
Let $\mu$ be a Borel probability measure on $\calX$ satisfying
\begin{align}\label{e:nu}
\int\rho(x,A)\,d\mu(x)\,\le\,L^2\mu(A),\qquad A\in\frakB(\calX),
\end{align}
with a constant $L\ge 0$, see \cite{PhilScha23b} for a detailed discussion. For $p\in [1,\infty)$, the {\em linear}, but infinite-dimensional Koopman operator $K_p : L^p(\mu)\to L^p(\mu)$ of the nonlinear process~$X_n$ is defined by the identity
\begin{equation}\label{e:koopman}
    (K_p\psi)(x) = \bE\big[\psi(X_1)\,|\,X_0=x\big] = \int\psi(y)\,\rho(x,dy)
\end{equation}
for all $\psi\in L^p(\mu)$.\footnote{In Appendix~\ref{a:just}, we show that $\psi \in L^p(\mu)$ is $\rho(x,\,\cdot\,)$-inte\-grable for $\mu$-a.e.\ $x\in\calX$. Hence, the Koopman operator is well defined. In fact, condition~\eqref{e:nu} is both necessary and sufficient for $K_p$ to be well-defined and bounded (with $\|K_p\| \le L^{2/p}$).} 
%In the special case $p=2$, we set $K := K_2$.
In the deterministic case~\eqref{eq:DTDS} with $\veps=0$, this reduces to $K_p\psi = \psi\circ T$. An iterative application of~\eqref{e:koopman} yields $(K_p^n \psi)(x) = \int \psi(y)\,\rho_n(x,dy)$ for $n\in\N$. We set $K := K_2$.

%\subsection{Extended dynamic mode decomposition}\label{ss:EDMD}
Let us briefly recall the well-known extended dynamic mode decomposition (EDMD, see~\citet{WillKevr15}), which aims at approximating the Koopman operator. To this end, %let $\mu\in\{\nu,\pi\}$ and 
let a dictionary $\calD = \{\psi_1,\ldots,\psi_N\}\subset L^2(\mu)$ of $\mu$-linearly independent (cf.\ Definition~\ref{d:lin_ind}) continuous functions on $\calX$ be given. 
For %If we define 
the $N$-dimensional subspace $\bV := \linspan\calD$ and matrices $C,C_+\in\R^{N\times N}$ given by
\[
    C = \big(\<\psi_i,\psi_j\>\big)_{i,j=1}^N
\quad\text{ and }\quad
C_+ = \big(\<\psi_i,K\psi_j\>\big)_{i,j=1}^N,
\]
$C$ is invertible and the matrix representation $K_\bV$ of the compression $P_\bV K|_\bV$ % of~$K$ to $\bV$ 
w.r.t.\ the basis $\calD$ is %given by
\begin{align}\label{e:compression}
K_\bV = C^{-1}C_+,
\end{align}
as rigorously shown in Lemma~\eqref{l:mat_repr}.
In EDMD, the matrix $K_\bV$ is learned by using evaluations of the dictionary observables on data samples $(x_k,y_k) \in \calX \times \calX$, $k \in [0:m-1]$, which are collected in the %In order to approximate the matrices $C$ and $C_+$, one defines the 
$N\times m$ data matrices
\[
\Psi_X = \big[\Psi(x_k)\big]_{k=0}^{m-1}\quad\text{ and }\quad \Psi_Y = \big[\Psi(y_k)\big]_{k=0}^{m-1},
\]
where $\Psi = [\psi_1,\ldots,\psi_N]^\top$. Then, the matrix representation of the Koopman operator estimator is given by
\begin{align}\label{e:koop_est}
    \wh K_m = \wh C^{-1}\wh C_+,
\end{align}
using %and uses these to define 
the empirical estimators of $C$ and $C_+$, respectively, i.e., $\wh C = \tfrac 1m\Psi_X\Psi_X^\top$ and $\wh C_+ = \tfrac 1m\Psi_X\Psi_Y^\top$.

%\subsection{Ergodic and i.i.d.\ sampling}

% from Section~\ref{sec:results} on. 
%In terms of sampling, we distinguish two cases.
In this work, we distinguish between two different sampling schemes.

\paragraph{(S1) Ergodic sampling $\mu = \pi$.}
We assume the existence of an invariant probability measure~$\pi$ for $X_n$, i.e.,
\begin{align}\label{e:inv}
    \int\rho(x,A)\,d\pi(x) = \pi(A),\qquad A\in\frakB(\calX).
\end{align}
In the deterministic case of~\eqref{eq:DTDS}, invariance of~$\pi$ corresponds to $\pi(T^{-1}(A)) = \pi(A)$ for all $A\in\frakB(\calX)$, i.e., $T$ is measure-preserving. Further, $\pi$ is assumed to be ergodic, i.e., whenever $A\in\frakB(\calX)$ is such that $\rho(x,A)=1$ for all $x\in A$, then $\pi(A)\in\{0,1\}$.

In this case, the Koopman operator $K_p$ is a contraction in~$L^p(\pi)$ and even an isometry in the deterministic case. 
The EDMD data consists of samples $x_k = X_k$ from a single trajectory of~$X_n$ with $x_0 \sim \pi$ and $y_k = x_{k+1}$. To ensure a.s.\ invertibility of~$\wh C$, we shall assume that for each $(N-1)$-dimensional subspace $M\subset\R^N$ and $x \in \Psi^{-1}(M)$, we have $\rho(x,\Psi^{-1}(M))=0$, see Appendix~\ref{a:proofs} for a proof of this sufficient condition.
%Further, the matrix~$\wh C$ is invertible a.s.\ if  
%This formula raises the immediate question under which conditions the random matrix $\wh C$ is invertible almost surely. The proofs of the next two lemmas and the definition of (strong) $\mu$-linear independence have been moved to Appendix .
%Hence, we assume that, for each $(N-1)$-dimensional subspace $M \subset \R^N$ and $x \in \Psi^{-1}(M)$, we have $\rho(x,\Psi^{-1}(M)) = 0$ 

\paragraph{(S2) I.i.d.\ sampling $\mu = \nu$.}
Let $\nu$ be any probability distribution on~$\calX$ satisfying \eqref{e:nu} with some constant $L\ge 0$. In this case, the EDMD data consists of i.i.d.\ samples $x_k \in \calX$ and $y_k|(x_k=x)\sim\rho(x,\,\cdot\,)$. Here, we assure the almost sure invertibility of $\wh C$ by assuming that $m \ge N$ and $\psi_1,\ldots,\psi_N$ are strongly $\mu$-linearly independent, see Appendix~\ref{a:proofs} for details including a proof of this characterization.

In the remainder of the manuscript, we tacitly assume
$$
\vphi := \sum_{j=1}^N\psi_j^2\in L^2(\mu).
$$

\section{Certifiable and efficient machine learning}\label{sec:results}
In this section, we provide a concise overview on our main results. In particular, we present novel 
%and sharp 
error bounds depending on the amount of data samples~$m$. The key tool enabling us to derive the error estimates and the respective convergence rates are 
%adroitly-composed 
exact formulae representing the variance of the EDMD estimators, which are shown --~together with an in-depth analysis~-- in the subsequent Section~\ref{sec:proofs}.

We motivate our findings by showing numerically approximated 
convergence rates while referring to Section~\ref{sec:numerics} for a detailed description of the numerical experiments. To be slightly more precise, we consider a reversible stochastic system modeling the folding kinetics of a
protein and a deterministic nonlinear partial differential equation
given by the Kuramoto-Sivashinsky equation for chaotic flame propagation.
In Figure~\ref{fig:spoiler}, we depict the convergence rate of the learning error $\|K_\bV - \wh K_m\|_F$ of EDMD in terms of the amount of data~$m$ used for ergodic sampling \textbf{(S1)}. We clearly observe that the convergence rate of the error is linear for the molecular dynamics example (note that we depict the root mean square error $[\bE[\|\wh C - C\|^2_F]]^{1/2}$ with corresponding rate $m^{-1/2}$) and superlinear for the deterministic nonlinear partial differential equation~(PDE).

\begin{figure}[htb]
    \centering
    \includegraphics[width=.35\columnwidth]{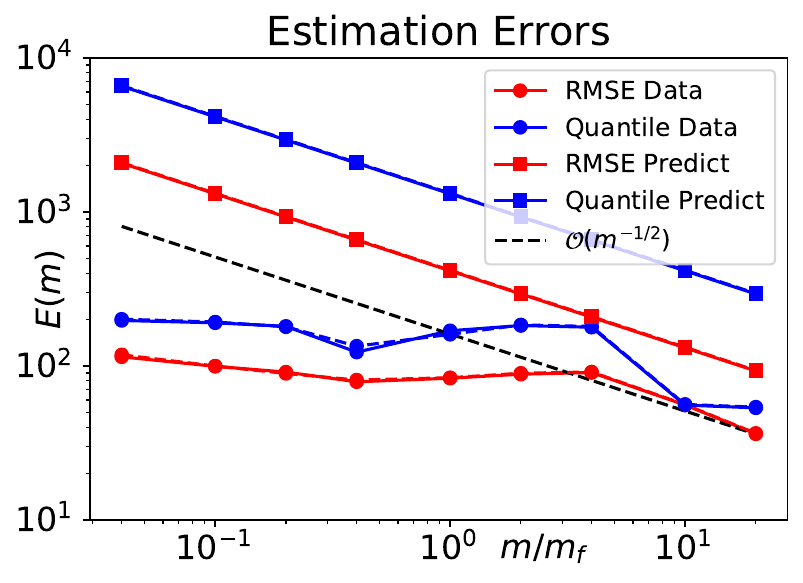}
    \hspace*{1cm}
    \includegraphics[width=.29\columnwidth]{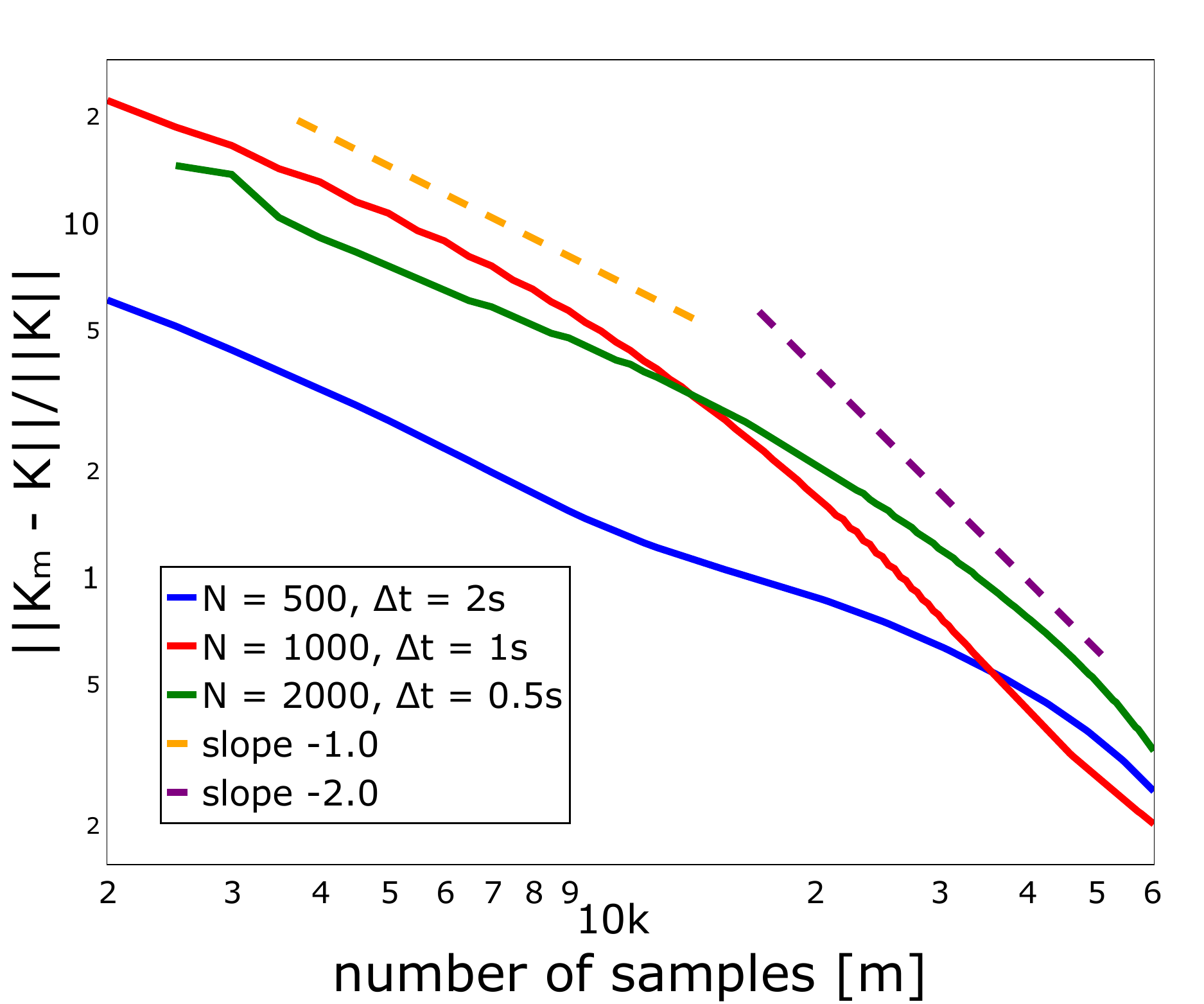}
    \caption{Convergence rate for ergodic sampling of a stochastic system for protein folding (left) and of a deterministic partial differential equation modeling chaotic flame propagation (right).}
    \label{fig:spoiler}
\end{figure}
\renewcommand{\arraystretch}{2.5}
\renewcommand\theadgape{\Gape[8pt]}
\renewcommand\cellgape{\Gape[8pt]}
\newcolumntype{?}{!{\vrule width 1.5pt}}

\begin{table*}[t]
\centering
\resizebox{.75\columnwidth}{!}{
  \begin{tabular}{|c c c|c|c?}
    \cline{1-4}
    &&\multicolumn{2}{|c|}{\large Sampling}\\
    \cline{3-4}
    && \multicolumn{1}{|c|}{\makecell{i.i.d.$\sim$\,probability distribution $\nu$ \\ $\int \rho(x,A)\,\mathrm{d}\nu(x) \leq L^2 \nu(A)$}}& \makecell{invariant and ergodic measure $\pi$\\$\int \rho(x,A)\,\mathrm{d}\pi(x) = \pi(A)$}\\
    \cline{1-4}
    \parbox[t]{2mm}{\multirow{2}{*}{\rotatebox[origin=c]{90}{\large Stochastic\hspace*{.1cm}}}}&\multicolumn{1}{|c|}{ \rotatebox[origin=c]{90}{\hspace{.1cm}DTDS}} & {\huge\checkmark}(Theorem~\ref{t:main1}) &  {\huge\checkmark}(Theorem~\ref{t:main2})\\
    \cline{2-4}
    &\multicolumn{1}{|c|}{\rotatebox[origin=c]{90}{\hspace{.15cm}SDE}} & \cite{NuskPeit23}: {\checkmark} & \makecell{\textcolor{gray}{\cite{NuskPeit23}: exp.~stability}\\ 
    {\huge\checkmark} (Theorem~\ref{t:main2})}\\
    \cline{1-5}\Cline{4-5}{1.5pt}
    \parbox[t]{2mm}{\multirow{2}{*}{\rotatebox[origin=c]{90}{\large Deterministic\hspace*{.1cm}}}}&\multicolumn{1}{|c|}{\rotatebox[origin=c]{90}{\hspace*{.1cm}\makebox{\makecell{DTDS\\ $(\varepsilon=0)$}}}} & {\huge\checkmark}(Theorem~\ref{t:main1}) & %
    \multicolumn{1}{?c|}{\makecell{\textcolor{gray}{\cite{Mezi22}: linear rate}\\ 
    {\huge\checkmark} (Theorem~\ref{thm:main_quadratic})}}
    & \parbox[t]{3mm}{\multirow{2}{*}{\rotatebox[origin=c]{90}{Superlinear rate }}}\\
    \cline{2-4}
    &\multicolumn{1}{|c|}{\rotatebox[origin=c]{90}{\makecell{SDE\\$(\sigma=0)$}}} & \makecell{ \cite{ZhanZuaz23}: {\checkmark}} & \multicolumn{1}{?c|}{{\huge \checkmark} (Theorem~\ref{thm:main_quadratic})}&\\
    \cline{1-5}\Cline{4-5}{1.5pt}
  \end{tabular}
  }
  \label{tab:cont}
  \caption{Fundamental contributions of this work.}
\end{table*}

The main contribution are rigorously derived %sharp 
error bounds for data-driven learning in the Koopman framework based on EDMD. As a byproduct, we obtain the first rigorous error analysis directly applicable to a broad class of highly-complex systems, which contains --~among many others~-- the two considered challenging applications. Our key findings are, in addition, summarized in Table~\ref{tab:cont}. 

\noindent \textbf{Case~(S1): Ergodic sampling of stochastic systems}. %, %e.g., \eqref{eq:DTDS} or \eqref{eq:SDE}, 
%see Subsection~\ref{subsec:ergodic}. % for details. 
For this sampling strategy and the special case~\eqref{eq:SDE}, error bounds with linear rate were already given in~\cite{NuskPeit23}. The proof relied on the assumption that the Koopman semigroup on $\one^\perp$ is exponentially stable. %which, in particular, \textcolor{red}{excludes the deterministic case}. 
Our first major result %given in 
Theorem~\ref{t:main2} does not require this assumption and provides a linear rate for a wide class of 
%stochastic
stochastic systems, including the molecular dynamics application as well as discrete-time systems like~\eqref{eq:DTDS}. We show that there is a constant $c \geq 0$ such that for all $\veps>0$,
\begin{align*}
    \bP\big(\|K_\bV - \wh K_m\|_F > \veps\big)\,\le\,\frac{c}{m\veps^2}
\end{align*}
holds. The corresponding convergence rate depicted in Figure~\ref{fig:spoiler} (left) reveals that this linear rate is sharp---in the sense that there are systems exposing the predicted worst-case linear convergence rate. %\red{[FN: appears to be tight? ``is sharp'' sounds very definitive]} \textcolor{magenta}{Ist insofern scharf, als wir zeigen, dass der Worst Case $c/(m\veps^2)$ ist und dieser auch ``angenommen'' wird.}

\textbf{Case~(S1): Ergodic sampling of deterministic systems}, e.g., \eqref{eq:DTDS} with $\varepsilon_n\equiv 0$ or \eqref{eq:SDE} with $\sigma \equiv 0$. %, see Subsection~\ref{subsec:ergodic_det} for details. 
The main condition for obtaining the above error bound for stochastic systems is that $\la=1$ is an isolated eigenvalue of the Koopman operator. Although this assumption is much more general than that of exponential stability of the Koopman semigroup on $\one^\perp$ imposed in \cite{NuskPeit23}, it excludes a broad class of {\em deterministic cases}, cf.\ \citet{KakuPete81}. However, leveraging advanced tools from operator theory, we prove the first superlinear convergence rates for EDMD with ergodic sampling of deterministic systems (such as nonlinear PDEs) as a second major result in Theorem~\ref{thm:main_quadratic}. That is, we prove that there are constants $c\geq 0$ and $\alpha\ge 1$ such that for all $\veps >0$
\begin{align*}
\bP\big(\|K_\bV-\wh K_m\|_F > \veps\big)\le\frac{c}{m^\alpha\veps^2}.
\end{align*}
This rate is observed in Figure~\ref{fig:spoiler} (right) illustrating EDMD for the Kuramoto-Sivashinsky equation.

\noindent \textbf{Case~(S2): i.i.d.~sampling} of, e.g., \eqref{eq:DTDS} or \eqref{eq:SDE}, see Section~\ref{sec:proofs} for details. The last major result considers the case of i.i.d.~sampling. Going beyond \citet{NuskPeit23}, and under very general assumptions, we provide an exponential convergence rate for EDMD with i.i.d.~sampling using Hoeffding's inequality in Theorem~\ref{t:main1}. We show that there are constants $c_1,c_2 >0$ such that for all $\veps>0$
\begin{align*}
\bP\big(\|K_\bV-\wh K_m\|_F > \veps\big)
\le c_1\exp(-c_2 m\veps^2).
\end{align*}

In the next section, %upcoming Section~\ref{sec:proofs}, 
we provide the precise statements of the presented error bounds. % presented above. Moreover, 
In Section~\ref{sec:numerics}, we revisit the molecular dynamics and flame propagation examples in %by means of 
an extensive case study. %following subsections, we shed light on the technical details of
%provide a rigorous analysis of 
%the particular error bounds.

\section{Convergence rates for EDMD-based machine learning} %on data-driven learning}}%of finite-data error bounds for learning}}
\label{sec:proofs}

Our strategy to prove the three estimates presented in Section~\ref{sec:results} consists of three major steps. First, we provide a representation formula for the variances of the empirical estimators $\widehat{C}$ and $\widehat{C}_+$ in terms of the $m$ sample points. Second, this representation is combined with concentration inequalities, i.e., Markov's and Hoeffding's inequality, to deduce a probabilistic bound on the errors $\hat{C}-C$ and $\widehat{C}_+-C_+$. In a last step, we combine these to a obtain a bound on the learning error $K_\bV - \widehat{K}_m$ in view of \eqref{e:compression} and \eqref{e:koop_est}.

\paragraph{Case~(S1): Ergodic sampling. Error bounds with non-restrictive %--~almost without 
assumptions.}%\label{subsec:ergodic}
%In this subsection, w
We let $\mu = \pi$ and draw the data samples from long ergodic trajectories of the process, according to case {\bf (S1)}.% We assume {\bf (A1)} and {\bf (A3)}.

A key step in our error analysis consists of deducing representations of the variances of $\wh C$ and $\wh C_+$ which are well-suited for further analysis. For the formulation of our next result, we note that always $K\one = \one$, hence, $K\linspan\{\one\}\subset\linspan\{\one\}$, and also $KL^2_0(\mu)\subset L^2_0(\mu)$, where $L^2_0(\mu) = \linspan\{\one\}^\perp$. In what follows, we set $K_0 := K|_{L^2_0(\mu)}$, which is a contractive linear operator from $L^2_0(\mu)$ into itself. Moreover, we define the constants
\[
\bE_+ := \<K\vphi,\vphi\> - \|C_+\|_F^2
\qquad\text{and}\qquad
\bE_0 := \|\vphi\|^2 - \|C\|_F^2.
\]

\begin{thm}\label{t:var}{\rm (Variance representation)}
Define the quantities
\begin{align*}
\sigma_{m,+}^2 := \bE_+ + \sum_{i,j=1}^N\<p_m(K_0)Qg_{ij},Qg_{ji}^*\>
\quad\text{and}\quad
\sigma_{m,0}^2 := \bE_0 + \sum_{i,j=1}^N\<K_0p_m(K_0)Q\psi_{ij},Q\psi_{ij}\>,
\end{align*}
where $\psi_{ij} = \psi_i\psi_j$, $g_{ij} = \psi_i\cdot K\psi_j$, $g_{ji}^* = \psi_j\cdot K^*\psi_i$, $Q = P_{L^2_0(\mu)}$, and $p_m$ is the polynomial
%\[
$p_m(z) = 2\sum_{k=1}^{m-1}(1-\tfrac km)z^{k-1}$.
%\]
Then the variances of $\wh C_+$ and $\wh C$ admit the following representations:
\begin{align*}
\bE\big[\|C_+ - \wh C_+\|_F^2\big] = \frac{\sigma_{m,+}^2}{m}
\qquad\text{and}\qquad
\bE\big[\|C - \wh C\|_F^2\big] = \frac{\sigma_{m,0}^2}{m}.
\end{align*}
\end{thm}

Next, a thorough analysis of the expressions $\sigma_{m,+}^2$ and $\sigma_{m,0}^2$ in the variance representations leads to the following bounds:
\begin{align*}
\sigma_{m,+}^2\le\left[1 + \|p_m(K_0)\|\right]\bE_{+}
\qquad\text{and}\qquad
\sigma_{m,0}^2\le\left[1 + \|K_0p_m(K_0)\|\right]\bE_{0}.
\end{align*}
So far, the results hold in full generality. However, if we further assume that $\la=1$ is an isolated simple eigenvalue\footnote{This condition is equivalent to $\la\notin\sigma(K_0)$.} of $K$, we may further estimate the above bounds independently of $m$:
\begin{align}
\begin{split}\label{e:sigma_est}
\sigma_{m,+}^2\le\left[1+4\|(I-K_0)^{-1}\|\right]\bE_{+}
\qquad\text{and}\qquad
\sigma_{m,0}^2\le\left[1+4\|K_0(I-K_0)^{-1}\|\right]\bE_{0}.
\end{split}
\end{align}
An application of Markov's inequality in combination with Lemma \ref{l:prob_absch} immediately yields the following theorem, which is the main result of this subsection.

\begin{thm}\label{t:main2}
Assume that $\la=1$ is an isolated simple eigenvalue of $K$. Then we have
\begin{align*}%\label{e:main}
\bP\big(\|C^{-1}C_+ - \wh C^{-1}\wh C_+\|_F > \veps\big)\,\le\,\frac{\alpha}{m\veps^2},
\end{align*}
with $\alpha$ provided in \eqref{e:alpha}. 
%For $\veps>0$, define the constant
In particular, if $\delta\in (0,1)$, then for $m\ge\frac{\alpha}{\delta\veps^2}$ ergodic samples, with probability at least $1-\delta$ we have that $\|C^{-1}C_+ - \wh C^{-1}\wh C_+\|_F \le \veps$.
\end{thm}

\begin{rem}
%{\bf (a)} 
If $K$ is normal (e.g., self-adjoint or unitary), then
%\begin{align*}%\label{e:dist}
$\|(I-K_0)^{-1}\| = \frac 1{\dist(1,\sigma(K_0))}$,    
%\end{align*}
where $\sigma(K_0)$ denotes the spectrum of the operator $K_0$. If there exist eigenvalues of $K$ close to $\la=1$, the above distance is small (hence $\|(I-K_0)^{-1}\|$ is large) and there exist so-called {\it meta-stable} sets, which are almost invariant; that is, trajectories of $X_n$ remain in these sets for a long time~\cite{Davies1982}. In this case, lots of measurements $m$ are needed to gather sufficient information on the process, which is reflected in Theorem \ref{t:main2}.

% \smallskip\noindent
% {\bf (b)} If there exist $M\ge 1$ and $\omega>0$ such that $\|K_0^n\|\le Me^{-\omega n}$ (which is equivalent to $K_0$ having spectral radius smaller than one\footnote{This condition was assumed in \citet{NuskPeit23} and in particular excludes the deterministic case.}), then, setting $q = e^{-\omega} < 1$, we have
% \begin{align*}
% \|p_m(K_0)\|
% &= 2\Bigg\|\sum_{k=1}^{m-1}\tfrac{m-k}{m}\cdot K_0^{k-1}\Bigg\|\le Mp_m(q)= \frac{2M}{1-q}\left(1 - \frac{1-q^m}{m(1-q)}\right)\,\le\,\frac {2M}{1-q}
% \end{align*}
% and $\|K_0p_m(K_0)\|\le\frac{2Mq}{1-q}$. Especially, if $M=1$, we obtain $\sigma_{m,+}^2\le\frac{3-q}{1-q}\cdot\bE_+$ and $\sigma_{m,0}^2\le\frac{1+q}{1-q}\cdot\bE_0$. For example, if $K$ is a non-negative self-adjoint operator with an isolated simple eigenvalue at $\la=1$, we have $M=1$ and $q = \max\sigma(K_0)$.
\end{rem}

\paragraph{Case~(S1): Ergodic sampling of deterministic systems. Superlinear convergence.}%\label{subsec:ergodic_det}
Let us consider the deterministic subcase of case {\bf (S1)}, where $K$ is a unitary composition operator with a bijective measure-preserving map $T : \calX\to\calX$, i.e., $Kf = f\circ T$. The key result is the next theorem, which shows that the variances of $\wh C_+$ and $\wh C$ exhibit a direct link to mean ergodicity.

\begin{thm}\label{l:fejer}{\rm (Variance representation)}
%Let $K$ be unitary. Then for the variances of $\wh C_+$ and $\wh C$, respectively, we have
If $K$ is unitary, for the variances of $\wh C_+$ and $\wh C$ we have
\begin{align*}
\bE\big[\|C_+ - \wh C_+\|_F^2\big] = \sum_{i,j=1}^N\bigg\|\frac 1m\sum_{k=0}^{m-1}K_0^kQg_{ij}\bigg\|^2
\quad\text{and}\quad
\bE\big[\|C - \wh C\|_F^2\big] = \sum_{i,j=1}^N\bigg\|\frac 1m\sum_{k=0}^{m-1}K_0^kQ\psi_{ij}\bigg\|^2.
\end{align*}
\end{thm}

By the mean ergodic theorem (see, e.g., \citet{Krengel85}), we know that for every single $f\in L^2_0(\mu)$ we have that $T_mf := \frac 1m\sum_{k=0}^{m-1}K_0^kf\to 0$ as $m\to\infty$ (in norm). However, this convergence can be arbitrarily slow and is, in addition, qualitatively bounded from above by $1/m$, which follows from \citet{ButzWest71} who proved that  $\|T_mf\| = o(1/m)$ implies $f=0$. Moreover, it will never be uniform\footnote{at least if $\mu$ is non-atomic} in the sense that $\|T_m\|\to 0$ as $m\to\infty$ (see, e.g., \citet{KakuPete81}). 
We may therefore {\em not} assume that $\la=1$ is an isolated simple eigenvalue of $K$, since otherwise $\|T_m\| = \|\frac 1m(I-K_0)^{-1}(I-K_0^m)\|\le \frac 2m\|(I-K_0)^{-1}\|\to 0$ as $m\to\infty$, which is a contradiction.

We thus have to impose assumptions on the functions that $T_m$ is applied to. For a function $f\in L^2(\mu)$ we let
\[
\mu_f(\Delta) := \|E(\Delta)f\|^2,\quad \Delta\in\frakB(\T),
\]
where $E$ denotes the spectral measure of the unitary operator $K$, cf.\ Appendix \ref{s:unitary}. Then $\mu_f$ is a finite measure on $\T$ describing the spectral distribution of $f$. Define the finite set of functions
\[
\calF := \{Q\psi_{ij} : i,j\in [1:N]\}\,\cup\,\{Qg_{ij} : i,j\in [1:N]\},
\]
the arcs $S_\theta := \{e^{2\pi it} : -\theta\le t\le\theta\}$, $\theta\in (0,1/2)$, and the constant
\[
M := \frac{8(1+\|C^{-1}\|_F^2\|C_+\|_F^2)^2}{\|C_+\|_F^2}\cdot\max\{\bE_0,\bE_+\}.
\]

\begin{thm}\label{thm:main_quadratic}
Assume that $K$ is unitary and suppose that there exist $\alpha\in (1,2)$, $\kappa\ge 0$, and $\theta\in (0,1/2)$ such that each $f\in\calF$ satisfies
\begin{align}\label{e:thin_meas}
\mu_f(S_{\gamma})\le\kappa\cdot\mu_f(S_\theta)\cdot\gamma^\alpha,\quad \gamma\in (0,\theta].
\end{align}
Then for $\veps\in (0,2)$ we have
\begin{align*}
\bP\big(\|C^{-1}C_+ - \wh C^{-1}&\wh C_+\|_F > \veps\big)\le\frac{C(\alpha,\kappa,\theta)M}{m^\alpha\veps^2}.
\end{align*}
with %The constant 
$C(\alpha,\kappa,\theta)$ %is 
provided in~\eqref{e:const}. If $\kappa=0$, then
%\begin{align*}
$\bP\big(\|C^{-1}C_+ - \wh C^{-1} \wh C_+\|_F > \veps\big)
\le\frac{M}{(1-\cos\theta)\cdot m^2\veps^2}$.
%\end{align*}
\end{thm}

Obviously, the factor $\mu_f(S_\theta)$ in \eqref{e:thin_meas} is due to normalization purposes. Hence, \eqref{e:thin_meas} means qualitatively that $\gamma^{-\alpha}\mu_f(S_\gamma) = O(1)$ (as $\gamma\to 0$). The following corollary shows that in the special case where $K$ has discrete spectrum, condition \eqref{e:thin_meas} is satisfied if the coefficients of $f$ have a certain decay.

\begin{cor}\label{c:thin_meas}
Assume that $K_0$ is of the form $K_0 = \sum_{n\in\N}e^{2\pi it_n}\<\,\cdot\,,f_n\>f_n$, where $t_n\in [-1/2,1/2)$, $n\in\N$, and $(f_n)$ is an orthonormal basis of $L_0^2(\mu)$. Then an $f\in L^2_0(\mu)$ satisfies \eqref{e:thin_meas} if $(\<f,f_n\>)_{n\in\N}\in\ell^2_w(\N)$, where $w_n = |t_n|^{-\alpha}$.
\end{cor}

% Let us briefly discuss the condition \eqref{e:thin_meas}. For this, assume that the spectrum of $K_0$ consists of eigenvalues $\la_n = e^{2\pi it_n}$ ($t_n\searrow 0$) with corresponding eigenfunctions $\vphi_n$. Then \eqref{e:thin_meas} is equivalent to $\sum_{n : t_n\le\gamma}|\<f,\vphi_n\>|^2\lesssim \gamma^\alpha$. As $\gamma\mapsto\gamma^\alpha$ is convex, this allows for relatively large gaps between the eigenvalues in relation to the coefficients $|\<f,\vphi_n\>|^2$ (in fact, $|\<f,\vphi_n\>|^2/(t_n-t_{n+1})\le t_n^{\alpha-1}$ is sufficient), while a behavior like $\sum_{n : t_n < \gamma}|\<f,\vphi_n\>|^2\sim\gamma^{1/2}$ forces the eigenvalues to be dense at $\la = 1$. Hence, \eqref{e:thin_meas} relates the coefficient decay with the position of the eigenvalues. Let us consider a simple example.

Since condition \eqref{e:thin_meas} seems very technical, we provide a simple example below.

\begin{ex}\label{ex:thin_meas}
Let $\calX = \T$ be the unit circle in $\C$ and consider the map $T : \calX\to\calX$, defined by $T(z) = cz$, where $c= e^{2\pi it_0}$, $t_0\in\R\backslash\Q$. Then the arc length measure $\mu$ on $\T$ is ergodic w.r.t.\ $T$, and the Koopman operator associated with $T$ satisfies
%\[
$Kf = \sum_{n\in\Z}c^n\<f,z^n\>z^n$,
%\]
where $\<\cdot\,,\cdot\>$ denotes the scalar product in $L^2(\mu)$. Hence, $K$ is unitary with eigenvalues $c^n$ and corresponding eigenfunctions $f_n(z) = z^n$. In particular, $K_0f = \sum_{n\in\Z^*}c^n\<f,z^n\>z^n$, where $\Z^* = \Z\backslash\{0\}$. Therefore, Corollary \ref{c:thin_meas} implies that functions $f\in L^2_0(\mu) = L^2(\T)\ominus\linspan\{\one\}$ whose sequence of Fourier coefficients $(\<f,z^n\>)_{n\in\Z}$ is contained in $\ell^2_w$, where $w_n = |(nt_0+\frac 12)\mod 1 - \frac 12|^{-\alpha}$, satisfy \eqref{e:thin_meas}. In particular, this is satisfied for trigonometric polynomials.
\end{ex}

\paragraph{Case~(S2): i.i.d.~sampling. Bounds on the EDMD estimation error}%\label{subsec:iid}

In this subsection, we let $\mu = \nu$ and draw i.i.d.\ data samples $(x_k,y_k)$ from $d\mu_{0} := \rho(\cdot,dy)\,d\mu$, according to case {\bf (S2)}. % We assume {\bf (A1)} and {\bf (A2)}.
Our main result in this case is the following theorem.

\begin{thm}\label{t:main1}
Assume that $C_+\neq 0$. Let $\veps>0$ and set $\sigma = 2\|C^{-1}\|_F\|C_+\|_F + \veps$. Then
\begin{align*}
\bP\big(\|C^{-1}C_+ &- \wh C^{-1}\wh C_+\|_F > \veps\big)\le\frac{\sigma^2}{m\veps^2}\left[\big(L\|C_+\|_F^{-2} + \|C^{-1}\|_F^2\big)\|\vphi\|^2 - 2\right].
\end{align*}
If, in addition, $\vphi\in L^\infty(\mu)$, then with $\tau = \sigma\|\vphi\|_\infty$ we have
\begin{align*}
\bP\big(\|C^{-1}C_+ - \wh C^{-1}\wh C_+\|_F > \veps\big)\le 2\exp\left(-\frac{m\veps^2\|C_+\|_F^2}{2\tau^2(1+L)^2}\right) + 2\exp\left(-\frac{m\veps^2}{8\tau^2\|C^{-1}\|_F^2}\right),
\end{align*}
\end{thm}

%Theorem \ref{t:main1} follows immediately from Proposition \ref{p:prob_est1} and Lemma \ref{l:prob_absch}.

\section{Numerical Examples}\label{sec:numerics}
In this part, we illustrate the deduced error bounds by means of two highly complex examples. The first is a protein folding simulation of a 35-amino acid chain, and the second is a nonlinear chaotic partial differential equation modeling flame propagation.

\paragraph{Molecular Dynamics Simulation}
We apply the Koopman approach to analyze the folding kinetics of the Fip35 WW-domain. This 35-amino acid protein has been used as a benchmarking system in many previous studies due to its small size and fast folding time scale of about ten micro-seconds. We use the $1.1$-milli-second molecular dynamics (MD) simulation data set published by D.E.\ Shaw research in~\citet{LINDORFF2011}. The underlying dynamics are a modified version of Hamiltonian dynamics, defined by a many-body potential energy function $V$ plus a thermostat, which ensures that the position space dynamics are effectively stochastic and reversibly sample the Boltzmann distribution $\exp(-\beta V(x))\,\mathrm{d}x$, enabling estimation of statistical averages via ergodic sampling. The complete data set comprises about $5.6 \cdot 10^6$ data points. Using straightforward parallelization on a small CPU cluster, the experiments below can be carried out in less than an hour.

We build our Koopman model on a 528-dimensional space of inter-atomic distances (closest heavy-atom inter-residue distances), which is a standard featurization for molecular simulation data sets. We use the {\rm mdtraj} library~\cite{mcgibbon15} to prepare the featurized data set. We employ a basis set of random Fourier features (RFFs)~\cite{Rahimi2007}, that is, complex plane waves of the forms
%\[
$\psi_j(x) = \exp(i\omega_j^\top x)$,
%\]
where $\omega_j$ are frequencies drawn from the spectral measure associated to a Gaussian radial basis function kernel. These features were shown to provide a fairly automatic way of generating an expressive dictionary in~\citet{nuske2023}. The Gaussian bandwidth parameter $\sigma$, the number of random features $p$, and the lag time $t$ for Koopman learning are tuned using the VAMP score metric~\cite{Wu2020}, see once again~\citet{nuske2023} for details. As shown in Figure~\ref{fig:fip35_md}, $\sigma = 30$, $p = 500$ and $t = 1\mu\mathrm{s}$ emerge as suitable parameters to accurately estimate the folding time scale. The leading eigenvectors of the Koopman model can be used to identify the folded and unfolded states from the simulation data, as illustrated in Figure~\ref{fig:fip35_md} C. The two leading eigenvectors of the Koopman model are transformed into membership functions $\chi_{0,1}$ by the PCCA method~\cite{Deuflhard2005}. A value close to one of these membership functions indicates that the system is in either the folded or unfolded state, as illustrated by representative structures on the left.

\begin{figure}[htb]
    \centering
    \includegraphics[width=0.38\textwidth]{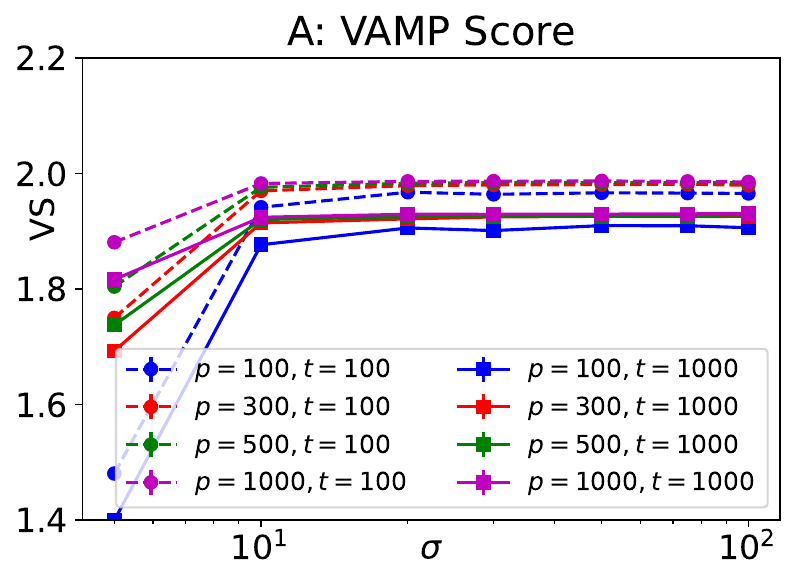}
    \hspace*{1cm}
    \includegraphics[width=0.38\textwidth]{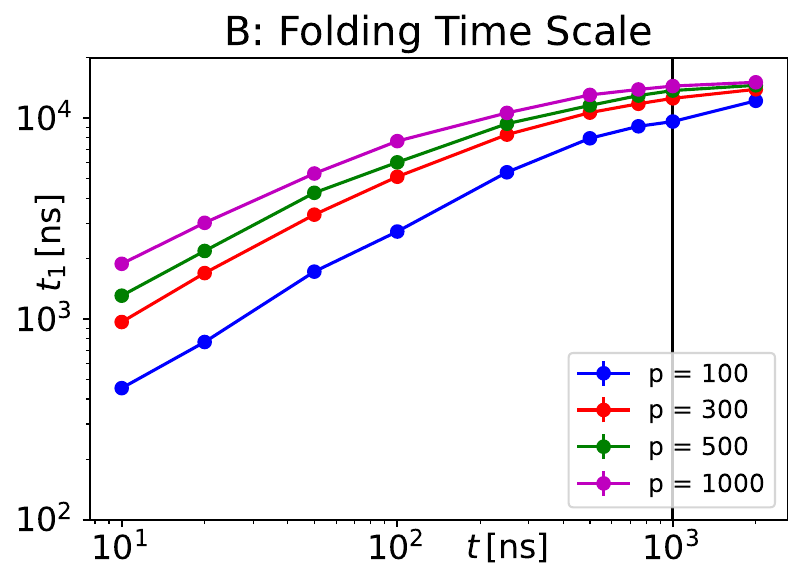}
    %\hspace*{1cm}
    %\includegraphics[width=.385\textwidth]{Protein_Structures.pdf}
    \caption{A: VAMP score metric as a function of kernel bandwidth $\sigma$ for different feature sizes $p$ and lag times $t$ (in  $\rm [ns]$). Regardless of the lag time, $p = 500$ and $\sigma = 30$ emerge as optimal. B: Estimated folding timescale for $\sigma = 30$ and different feature sizes $p$ as a function of lag time. It can be seen that convergence is achieved for $t = 1\,\mu\mathrm{s}$ and $p \geq 500$.}
    \label{fig:fip35_md}
\end{figure}

\begin{figure}[htb]
    \centering
    \includegraphics[width=.5\textwidth]{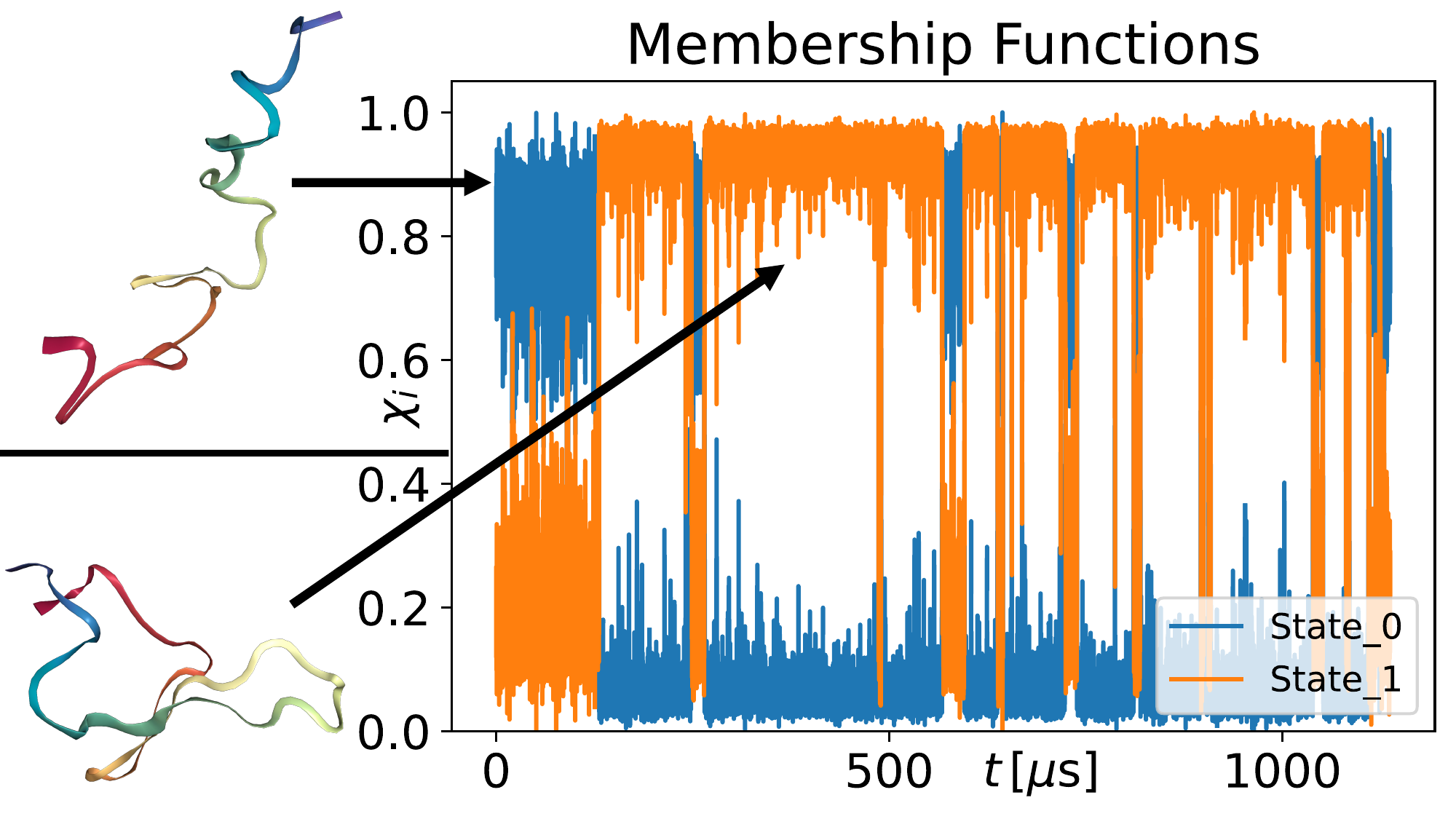}
    \caption{Identification of the folded state of Fip35 based on the Koopman model.}
    \label{fig:fip35_pcca}
\end{figure}

As the MD simulation is reversible, we can apply the error bounds in Eq.~\eqref{e:sigma_est}. Computing the exact asymptotic variances $\sigma_{m,0}^2$ and $\sigma_{m,+}^2$ via Theorem~\ref{t:var} and Eq.~\eqref{e:sigma_est} requires access to the true eigenvalues and Galerkin matrices $C$, $C_+$. Since these are not available, we estimate them by learning a reference model on all available data points. We then compute the error relative to this reference model if the Galerkin matrices are estimated using fewer data points, with $m$ ranging from $m = 2000$ to $m = 2\cdot 10^6$. 
In Figure~\ref{fig:spoiler} (left), we show %both 
the root mean square errors $[\bE[\|\wh C - C\|^2_F]]^{1/2}$ (red) and the 90\% error quantile % of the error 
(blue), estimated by our theoretical bounds in Eq.~\eqref{e:sigma_est} (squares) vs.~data-based estimates relative to the reference model (circles) using~$50$~different sub-samples of the data set, each of size~$m$. Throughout, % the figure, 
solid lines refer to~$C$, while dashed lines refer to $C_+$, but the results are almost indistinguishable. % for most quantities. 
The black line indicates a qualitative decay of the form $cm^{-1/2}$, where the pre-factor $c$ is the average ratio of the data-based RMSE over $m^{-1/2}$, for all values of $m$ greater than the folding time scale. The values of $m$ on the horizontal axis are normalized against the number of data points $m_f$ required to reach the folding time scale $10\,\mu\mathrm{s}$. Compared to the theoretical results, the actual error is about an order or magnitude smaller. We notice however, that after an initial period of about the length of the folding time scale,  the asymptotic decay of the error is well-described by our estimates (see the black line in the left plot of Figure~\ref{fig:spoiler}).

\paragraph{Nonlinear PDE.}
In our second example, we study the Kuramoto-Sivashinsky equation in two space dimensions, which is a widely-studied deterministic PDE modeling the dynamics of chaotic flame front propagation:
\begin{equation}\label{eq:KS2D}
    \partial_t x + \nabla^2 x + \nabla^4 x + |\nabla x|^2 = 0.
\end{equation}
The system state $x: \Omega \times [0,\infty)\rightarrow \R$  (cf.\ Figure~\ref{fig:KS2D} left) %for an example snapshot of a state on the attractor) 
depends on both space and time, and we consider a rectangular domain $\Omega = [-30\pi,30\pi]^2$ with periodic boundary conditions.

\begin{figure}[htb]
    \centering
    \includegraphics[width=.3\columnwidth]{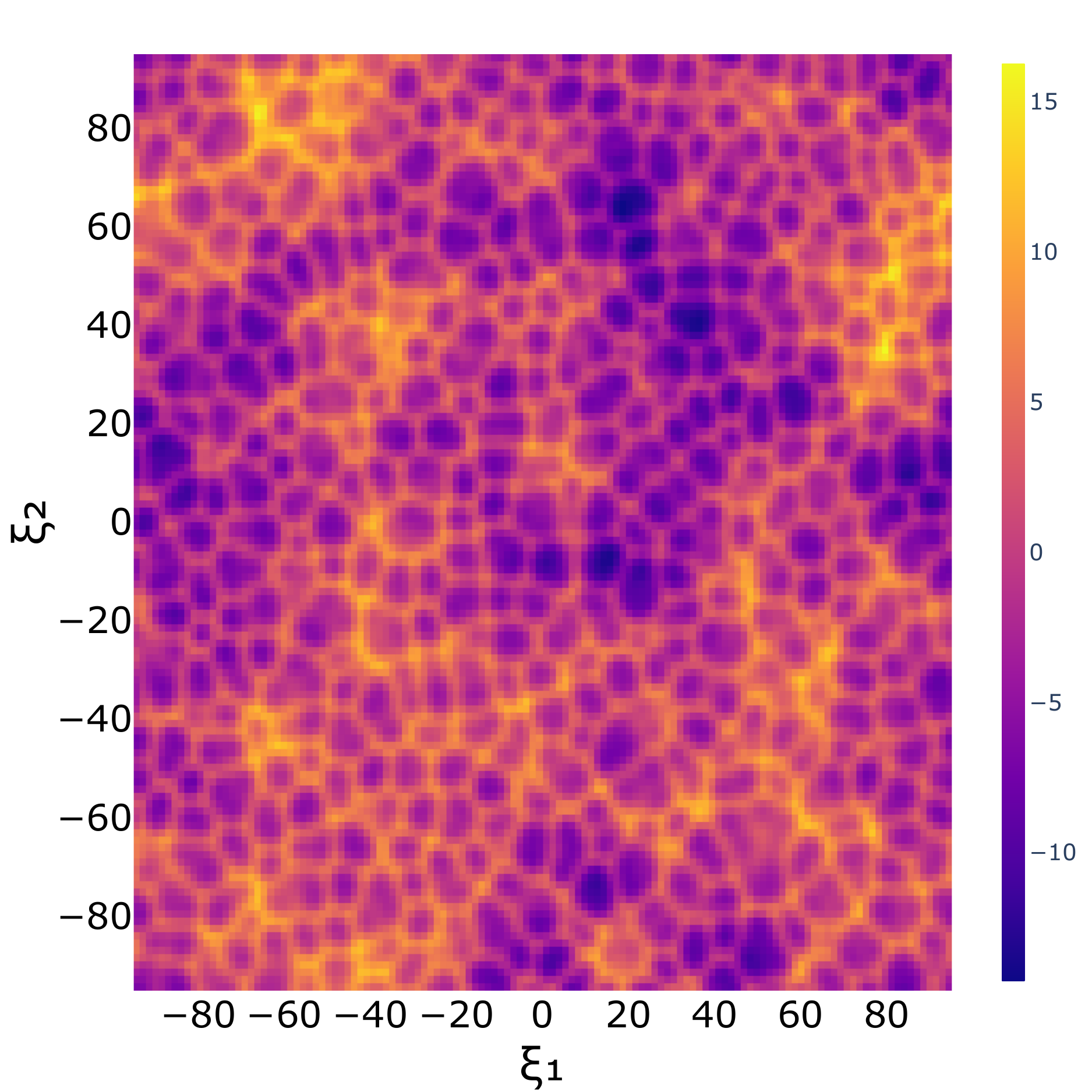}
    \hspace*{2cm}
    \includegraphics[width=.3\columnwidth]{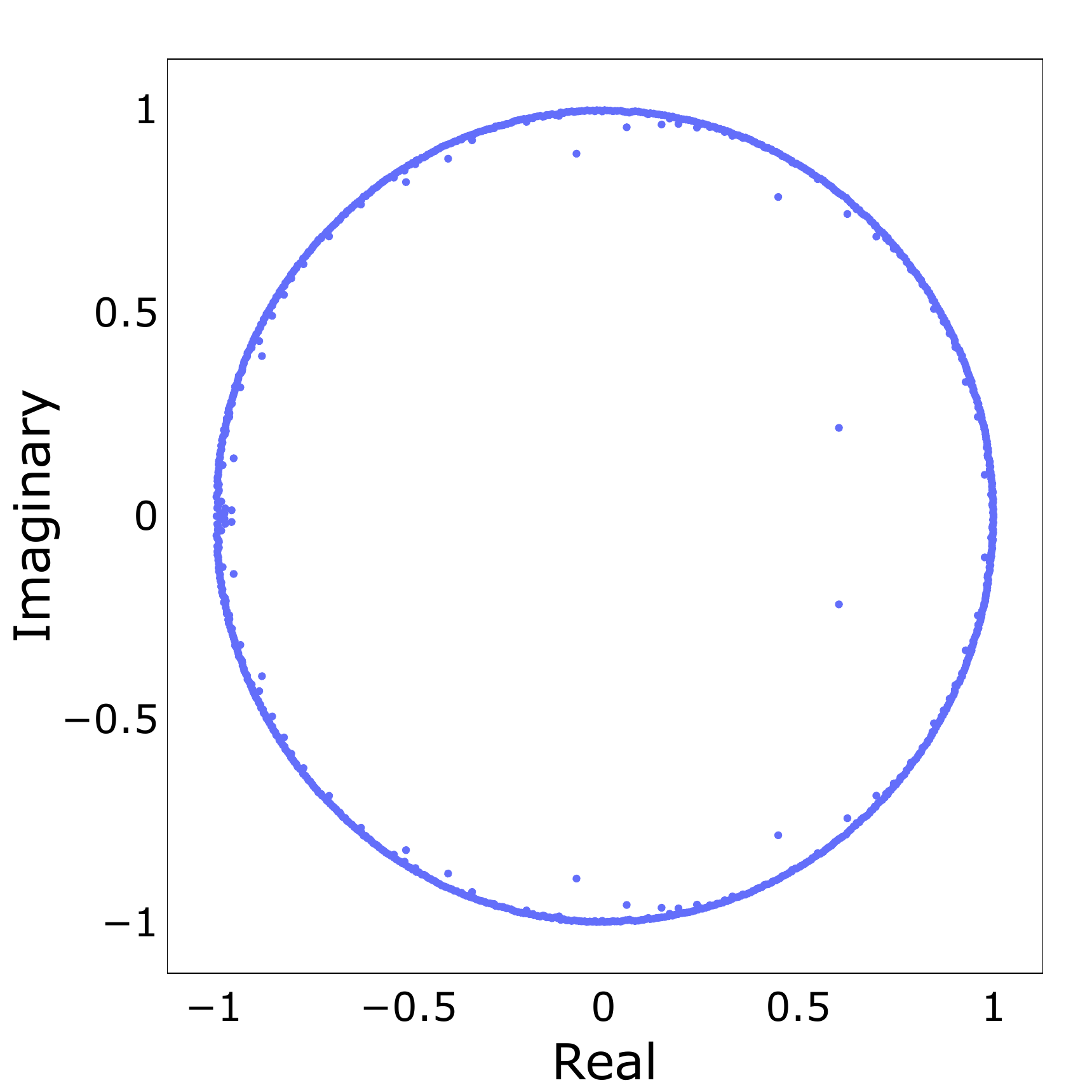}
    \caption{Left: Exemplary snapshot of a state~$x(\xi,t)$ on the attractor governed by the Kuramoto-Sivashinsky equation~\eqref{eq:KS2D}. Right: Eigenvalues of the EDMD matrix $\wh K_M$}
    \label{fig:KS2D}
\end{figure}

The numerical simulation is realized using the open-source code \textit{shenfun} \cite{Mor18}. For the used domain size, the system exhibits chaotic dynamics, which means that we can use ergodic sampling from a single long trajectory that covers the entire chaotic attractor. 
In our experiments, we collect a total of $M=100,000$ samples with a time step of $h = 0.01$. 
The spectral element discretization in space yields snapshots of dimension $128 \times 128 = 16,384$, which we reduce to $64 \times 64 = 4,096$ for computational reasons. As this is still prohibitively large for classical dictionaries, we rely on the kernel variant of EDMD \cite{WRK16}. We choose a dictionary $\calD$ consisting of $N=1000$ radial basis functions with a twice differentiable Matérn kernel (of order $\nu=2.5$), centered at $N$ spatially equidistant grid points, which means that our Koopman approximation is of dimension $\wh K_m\in\R^{N\times N}$. To estimate $\wh K_m$ we choose a lag time of $\Delta t = 1$, and as we do not have access to the true Koopman compression $K_\bV$, we study the behavior of the relative error between $\wh K_m$ and $\wh K_M$ for increasing $m$.
Figure \ref{fig:KS2D} (right) shows the eigenvalue spectrum of the matrix $\wh K_M$ for $N=1000$ radial basis functions and all $M$ samples. As usual %is very common 
for chaotic systems (cf.\ \citet{RMB+09}) and proven in Section~\ref{sec:proofs}, the eigenvalues are distributed %all 
around the unit circle. Consequently, %we see that 
the eigenvalue $\lambda=1$ is \emph{not} isolated and thus, according to Theorem \ref{thm:main_quadratic}, the convergence rate with respect to $m$ should range somewhere between one and two. When taking a closer look at Figure~\ref{fig:spoiler} (right), this is exactly the behavior that we observe.

\section{Conclusion}
We provided novel error bounds on EDMD for both ergodic and i.i.d.\ sampling for discrete- and continuous-time nonlinear dynamical systems under non-restrictive conditions. %rather mild assumptions. %As it turns out, f
For i.i.d.\ sampling, we established convergence at an 
exponential rate. % can be expected. 
Contrary, %On the other hand, 
the convergence of EDMD with ergodic sampling depends on whether the underlying system is stochastic or deterministic. While the error decays at a linear rate in the first case, we show that a broad class of deterministic systems exhibits superlinear convergence. Our theoretical results are underpinned by %means of 
numerical simulations of highly-complex %dynamical 
systems such as the folding kinetics of a protein in molecular dynamics and the chaotic nonlinear Kuramoto-Sivashinsky equation. In future research, we investigate suitable choices of the dictionary to speed up the convergence according to the newly identified conditions of Theorem~\ref{thm:main_quadratic} %will investigate the impact of oscillation of the dictionary observables on the EDMD performance 
in the deterministic case.

\section{Acknowlegdment}
We thank {\em D.E. Shaw Research} for providing the Fip35 simulation data set.

\section{Conclusion}
We provided novel error bounds on EDMD for both ergodic and i.i.d.\ sampling for discrete- and continuous-time nonlinear dynamical systems under non-restrictive conditions. For i.i.d.\ sampling, we established convergence at an 
exponential rate. Contrary, the convergence of EDMD with ergodic sampling depends on whether the underlying system is stochastic or deterministic. While the error decays at a linear rate in the first case, we show that a broad class of deterministic systems exhibits superlinear convergence. Our theoretical results are underpinned by numerical simulations of highly-complex systems such as the folding kinetics of a protein in molecular dynamics and the chaotic nonlinear Kuramoto-Sivashinsky equation. In future research, we investigate suitable choices of the dictionary to speed up the convergence according to the newly identified conditions of Theorem~\ref{thm:main_quadratic} in the deterministic case.

%\bibliography{MyRefs}
%\bibliographystyle{MyBibStyle}

%%%%%%%%%%%%%%%%%%%%%%%%%%%%%%%%%%%%%%%%%%%%%%%%%%%%%%%%%%%%%%%%%%%%%%%%%%%%%%%
%%%%%%%%%%%%%%%%%%%%%%%%%%%%%%%%%%%%%%%%%%%%%%%%%%%%%%%%%%%%%%%%%%%%%%%%%%%%%%%
% APPENDIX
%%%%%%%%%%%%%%%%%%%%%%%%%%%%%%%%%%%%%%%%%%%%%%%%%%%%%%%%%%%%%%%%%%%%%%%%%%%%%%%
%%%%%%%%%%%%%%%%%%%%%%%%%%%%%%%%%%%%%%%%%%%%%%%%%%%%%%%%%%%%%%%%%%%%%%%%%%%%%%%
\newpage
\appendix
\onecolumn
\section{Proofs of the main results}
In this section, we provide proofs for our main theorems. For the convenience of the reader, we repeat the particular results.

\subsection{Proof of Theorem \ref{t:main2} in Section~\ref{sec:proofs}}
In this subsection, we let $\mu = \pi$ be an ergodic and invariant measure for the process $X_n$, i.e., case \textbf{(S1)}. Furthermore, we assume that for each $(N-1)$-dimensional subspace $M\subset\R^N$ and $x\in\Psi^{-1}(M)$ we have $\rho(x,\Psi^{-1}(M))=0$ in order to guarantee the a.s.\ invertibility of $\wh C$, see Lemma \ref{l:case2}.

\textbf{Theorem~\ref{t:main2}.}
\textit{Assume that $\la=1$ is an isolated simple eigenvalue of $K$. For $\veps>0$, define the constant
\begin{align}\label{e:alpha}
\alpha := \big[1+4\|(I-&K_0)^{-1}\|\big]\cdot\big[2\|C^{-1}\|_F\|C_+\|_F + \veps\big]^2\cdot\left[\big(\|C_+\|_F^{-2} + \|C^{-1}\|_F^2\big)\|\vphi\|^2 - 2\right].
\end{align}
Then we have
\begin{align*}
\bP\big(\|C^{-1}C_+ - \wh C^{-1}\wh C_+\|_F > \veps\big)\,\le\,\frac{\alpha}{m\veps^2}.
\end{align*}
In particular, if $\delta\in (0,1)$, then for $m\ge\frac{\alpha}{\delta\veps^2}$ ergodic samples, with probability at least $1-\delta$ we have that $\|C^{-1}C_+ - \wh C^{-1}\wh C_+\|_F \le \veps$.
}
\begin{proof}
In a first step, we provide exact formulae for the variances of $\wh C$ and $\wh C_+$ in Theorem~\ref{t:var}. Next, we find bounds on these variances in Proposition \ref{p:sig2} and apply Markov's inequality to deduce a probabilistic bound on the errors $\|C_+ - \wh C_+\|_F^2$ and $\|C - \wh C\|_F^2$. Then Lemma \ref{l:prob_absch} immediately yields the desired bound for $\|C^{-1}C_+ - \wh C^{-1} \wh C_+\|^2_F$.
\end{proof}

For completeness, we repeat Theorem \ref{t:var} here. To this end, recall the definitions of $\bE_+$ and $\bE_0$:
\[
\bE_+ := \<K\vphi,\vphi\> - \|C_+\|_F^2
\qquad\text{and}\qquad
\bE_0 := \|\vphi\|^2 - \|C\|_F^2.
\]

%\begin{thm}\label{t:var_app}
\textbf{Theorem~\ref{t:var}.}
\textit{Define the quantities
\begin{align*}
\sigma_{m,+}^2 := \bE_+ + \sum_{i,j=1}^N\<p_m(K_0)Qg_{ij},Qg_{ji}^*\>
\qquad\text{and}\qquad
\sigma_{m,0}^2 := \bE_0 + \sum_{i,j=1}^N\<K_0p_m(K_0)Q\psi_{ij},Q\psi_{ij}\>,
\end{align*}
where $\psi_{ij} = \psi_i\psi_j$, $g_{ij} = \psi_i\cdot K\psi_j$, $g_{ji}^* = \psi_j\cdot K^*\psi_i$ ($i,j=1,\ldots,N$), $Q = P_{L^2_0(\mu)}$, and $p_m$ is the polynomial
\[
p_m(z) = 2\sum_{k=1}^{m-1}(1-\tfrac km)z^{k-1}.
\]
Then the variances of $\wh C_+$ and $\wh C$ admit the following representations:
\begin{align}
\bE\big[\|C_+ - \wh C_+\|_F^2\big] &= \frac{\sigma_{m,+}^2}{m},\label{e:var2}\\
\bE\big[\|C - \wh C\|_F^2\big] &= \frac{\sigma_{m,0}^2}{m}.\label{e:var3}
\end{align}}
%\end{thm}
\begin{proof}
With $\Phi_k := \Psi(x_k)\Psi(x_{k+1})^\top$, $k=0,\ldots,m-1$, we have
\begin{align*}
\bE\big[\|C_+ - \wh C_+\|_F^2\big]
&= \bE\bigg[\bigg\|\frac 1m\sum_{k=0}^{m-1}\big[C_+ - \Phi_k\big]\bigg\|_F^2\bigg] = \frac 1{m^2}\bE\bigg[\sum_{k=0}^{m-1}\sum_{\ell=0}^{m-1}\big\<C_+ - \Phi_k,C_+ - \Phi_\ell\big\>_F\bigg]\\
&= \frac 1{m^2}\bE\bigg[\sum_{k=0}^{m-1}\big\|C_+ - \Phi_k\big\|_F^2 + 2\sum_{k=1}^{m-1}\sum_{\ell=0}^{k-1}\big\<C_+ - \Phi_k,C_+ - \Phi_\ell\big\>_F\bigg]\\
&= \frac 1{m^2}\sum_{k=0}^{m-1}\bE\big[\|C_+ - \Phi_k\|_F^2\big] + \frac 2{m^2}\sum_{k=1}^{m-1}(m-k)\bE\big[\big\<C_+ - \Phi_k,C_+ - \Phi_0\big\>_F\big]\\
&= \frac 1m\big(\bE\big[\|\Phi_0\|_F^2\big] - \|C_+\|_F^2\big) + \frac 2m\sum_{k=1}^{m-1}(1 - \tfrac km)\bE\big[\big\<C_+ - \Phi_k,C_+ - \Phi_0\big\>_F\big].
\end{align*}
Concerning the first summand, we compute
\begin{align}\label{e:krank}
\|\Phi_0\|_F^2 = \big\|\big(\psi_i(x_0)\psi_j(x_1)\big)_{i,j=1}^N\big\|_F^2 = \sum_{i,j=1}^N\psi_i^2(x_0)\psi_j^2(x_1) = \vphi(x_0)\vphi(x_1)
\end{align}
and hence
\begin{align*}
\bE\big[\|\Phi_0\|_F^2\big]
= \int\int\vphi(x)\vphi(y)\,\rho(x,dy)\,d\mu(x) = \int\vphi(x)(K\vphi)(x)\,d\mu(x) = \big\<K\vphi,\vphi\big\>.
\end{align*}
Thus,
\begin{align*}
\bE\big[\|C_+ - \wh C_+\|_F^2\big] = \frac 1m\big[\<K\vphi,\vphi\> - \|C_+\|_F^2\big] + \frac 2m\sum_{k=1}^{m-1}(1 - \tfrac km)\bE\big[\big\<C_+ - \Phi_k,C_+ - \Phi_0\big\>_F\big].
\end{align*}
In the case of ergodic sampling, the cross terms in the second summand do not vanish. First,
\begin{align*}
\bE\big[\big\<C_+ - \Phi_k,C_+ - \Phi_0\big\>_F\big] + \|C_+\|_F^2
&= \bE\big[\big\<\Phi_k,\Phi_0\big\>_F\big]
\end{align*}
and next,
\begin{align*}
\bE\big[\big\<\Phi_k,\Phi_0\big\>_F\big]
&= \sum_{i,j=1}^N\bE\big[\psi_i(x_k)\psi_j(x_{k+1})\psi_i(x_0)\psi_j(x_1)\big]\\
&= \sum_{i,j=1}^N\int\int\int\int\psi_i(x')\psi_j(y')\psi_i(x)\psi_j(y)\,\rho(x',dy')\,\rho_{k-1}(y,dx')\,\rho(x,dy)\,d\mu(x)\\
&= \sum_{i,j=1}^N\int\int\int\psi_i(x')\psi_i(x)\psi_j(y)(K\psi_j)(x')\,\rho_{k-1}(y,dx')\,\rho(x,dy)\,d\mu(x)\\
&= \sum_{i,j=1}^N\int\int\psi_i(x)\psi_j(y)\big(K^{k-1}g_{ij}\big)(y)\,\rho(x,dy)\,d\mu(x)\\
&= \sum_{i,j=1}^N\int\psi_i(x)\big(K_1[\psi_j\cdot K^{k-1}g_{ij}]\big)(x)\,d\mu(x)\\
&\overset{(*)}{=} \sum_{i,j=1}^N\big\<\psi_j\cdot K^*\psi_i,K^{k-1}g_{ij}\big\> = \sum_{i,j=1}^N\big\<g_{ji}^*,K^{k-1}g_{ij}\big\>.
\end{align*}
For a justification of $(*)$ see Lemma \ref{l:geht} in the Appendix.

Let $P := P_{\one} = \<\,\cdot\,,\one\>\one$. Then, since $\int g_{ji}^*\,d\mu = \int g_{ij}\,d\mu$,
\begin{align*}
\sum_{i,j=1}^N\big\<Pg_{ji}^*,K^{k-1}Pg_{ij}\big\>
&= \sum_{i,j=1}^N\big\<\<g_{ji}^*,\one\>\one,\<g_{ij},\one\>\one\big\> = \sum_{i,j=1}^N\left|\int g_{ij}\,d\mu\right|^2 = \sum_{i,j=1}^N|\<\psi_i,K\psi_j\>|^2 = \|C_+\|_F^2.
\end{align*}
Similarly, we get
\begin{align}\label{e:later}
\sum_{i,j=1}^N\|Pg_{ji}^*\|^2 = \sum_{i,j=1}^N\|Pg_{ij}\|^2 = \|C_+\|_F^2.
\end{align}
Therefore,
\[
\bE\big[\big\<C_+ - \Phi_k,C_+ - \Phi_0\big\>_F\big] = \sum_{i,j=1}^N\big\<K_0^{k-1}Qg_{ij},Qg_{ji}^*\big\>,
\]
and \eqref{e:var2} is proved.

The proof of \eqref{e:var3} is similar. Here, we set $\wt\Phi_k = \Psi(x_k)\Psi(x_k)^\top$ and observe that
\begin{align*}
\bE\big[\|C - \wh C\|_F^2\big] = \frac 1m\big[\bE[\|\wt\Phi_0\|_F^2] - \|C_+\|_F^2\big] + \frac 2m\sum_{k=1}^{m-1}(1 - \tfrac km)\big(\bE\big[\big\<\wt\Phi_k,\wt\Phi_0\big\>_F\big] - \|C\|_F^2\big).
\end{align*}
Next,
\begin{align*}
\bE\big[\big\<\wt\Phi_k,\wt\Phi_0\big\>_F\big]
&= \sum_{i,j=1}^N\bE\big[\psi_i(x_k)\psi_j(x_k)\psi_i(x_0)\psi_j(x_0)\big] = \sum_{i,j=1}^N\bE\big[\psi_{ij}(x_k)\psi_{ij}(x_0)\big]\\
&= \sum_{i,j=1}^N\int\int\psi_{ij}(x)\psi_{ij}(y)\,\rho_k(x,dy)\,d\mu(x)\\
&= \sum_{i,j=1}^N\int\psi_{ij}(x)(K^k\psi_{ij})(x)\,d\mu(x)
= \sum_{i,j=1}^N\big\<K^k\psi_{ij},\psi_{ij}\big\>.
\end{align*}
In particular, $\bE\big[\big\|\wt\Phi_0\|_F^2\big] = \sum_{i,j=1}^N\|\psi_{ij}\|^2 = \|\vphi\|^2$. Finally, making use of
\begin{align}\label{e:later2}
\sum_{i,j=1}^N\big\<K^kP\psi_{ij},P\psi_{ij}\big\> = \sum_{i,j=1}^N\|P\psi_{ij}\|^2 = \|C\|_F^2,    
\end{align}
we obtain \eqref{e:var3}.
\end{proof}

\begin{prop}\label{p:sig2}
We have
\begin{align*}
\sigma_{m,+}^2&\le\left[1 + \|p_m(K_0)\|\right]\bE_{+}\\
\sigma_{m,0}^2&\le\left[1 + \|K_0p_m(K_0)\|\right]\bE_{0}.
\end{align*}
If, in addition, $\la=1$ is an isolated simple eigenvalue of $K$, then
\begin{align*}
\sigma_{m,+}^2&\le\left[1+4\|(I-K_0)^{-1}\|\right]\bE_{+},\\
\sigma_{m,0}^2&\le\left[1+4\|K_0(I-K_0)^{-1}\|\right]\bE_{0}.
\end{align*}
\end{prop}
\begin{proof}
Let again $P := P_\one = \<\,\cdot\,,\one\>\one$. Then
\begin{align*}
\sum_{i,j=1}^N\<p_m(K_0)Qg_{ij},Qg_{ji}^*\>
&\le \left(\sum_{i,j=1}^N\|p_m(K_0)\|^2\|Qg_{ij}\|^2\right)^{1/2}\left(\sum_{i,j=1}^N\|Qg_{ji}^*\|^2\right)^{1/2}\\
&= \|p_m(K_0)\|\left(\sum_{i,j=1}^N\big[\|g_{ij}\|^2 - \|Pg_{ij}\|^2\big]\right)^{1/2}\left(\sum_{i,j=1}^N\big[\|g_{ji}^*\|^2 - \|Pg_{ji}^*\|^2\big]\right)^{1/2}.
\end{align*}
Now, as $(Kf)^2\le Kf^2$ for $f\in L^4(\mu)$,
\[
\sum_{i,j=1}^N\|g_{ij}\|^2 = \sum_{i,j=1}^N\int\psi_i^2(K\psi_j)^2\,d\mu\,\le\,\sum_{i,j=1}^N\int\psi_i^2\cdot K\psi_j^2\,d\mu = \<K\vphi,\vphi\>.
\]
Since we also have $(K^*f)^2\le K^*f^2$ for $f\in L^4(\mu)$ (cf.\ Appendix \ref{a:just}), we similarly conclude $\sum_{i,j=1}^N\|g_{ji}^*\|^2\le\<K\vphi,\vphi\>$ and thus (cf.\ \eqref{e:later})
\[
\sum_{i,j=1}^N\<p_m(K_0)Qg_{ij},Qg_{ji}^*\>\,\le\,\|p_m(K_0)\|\big(\<K\vphi,\vphi\> - \|C_+\|_F^2\big) = \|p_m(K_0)\|\cdot\bE_+.
\]
Similarly (cf.\ \eqref{e:later2}),
\begin{align*}
\sum_{i,j=1}^N\<K_0p_m(K_0)Q\psi_{ij},Q\psi_{ij}\>
&\le \|K_0p_m(K_0)\|\sum_{i,j=1}^N\big[\|\psi_{ij}\|^2 - \|P\psi_{ij}\|^2\big] = \|K_0p_m(K_0)\|\cdot\bE_0.
\end{align*}
Assume now that $\la=1$ is an isolated simple eigenvalue of $K$. Since
\[
p_m(z) %= \frac{2}{1-z}\left(1 - \frac{1-z^m}{m(1-z)}\right) 
= \frac{2}{1-z}\bigg(1 - \frac{1}{m}\sum_{k=0}^{m-1}z^k\bigg),\qquad z\in\C\backslash\{0\},
\]
making use of $\|K_0\|\le 1$, we observe that
\begin{align*}
\|p_m(K_0)\| = \bigg\|2(I-K_0)^{-1}\Big(I - \tfrac 1m\sum_{k=0}^{m-1}K_0^k\Big)\bigg\|\,\le\,4\|(I-K_0)^{-1}\|.
% = 2(I-K_0)^{-2}\big((1-\tfrac 1m)I - K_0 + \tfrac 1mK_0^m\big).
\end{align*}
Similarly, $\|K_0p_m(K_0)\|\le 4\|K_0(I-K_0)^{-1}\|$.
%Moreover, $\|I-K_0\|\le 2$, so that
%\[
%1 = \|(I-K_0)(I-K_0)^{-1}\|\le\|I-K_0\|\|(I-K_0)^{-1}\|\le 2\|(I-K_0)^{-1}\|,
%\]
%which concludes the proof of the proposition.
\end{proof}

\subsection{Proof of Theorem \ref{thm:main_quadratic} in Section \ref{sec:proofs}}\label{ss:det_app}
Let us consider the deterministic subcase of ergodic sampling ,i.e., case {\bf (S1)}, where $K$ is a composition operator with a measure preserving map $T : \calX\to\calX$. Note that, in this case, we have $K_1(fg) = Kf\cdot Kg$ for $f,g\in L^2(\mu)$. Let us further assume that $T$ is bijective, so that the corresponding Koopman operator $K$ is unitary. Hence, there exists a bounded self-adjoint operator $A$ on $L^2(\mu)$ with $\sigma(A)\subset [-\pi,\pi]$ such that $K = e^{iA}$. We set $A_0 := A|_{L^2_0(\mu)}$.

The following lemma is a slightly extended version of Lemma \ref{l:fejer}.

%\begin{lem}\label{l:fejer_app}
\textbf{Lemma~\ref{l:fejer}} (Extended version).
\textit{
For the variances of $\wh C_+$ and $\wh C$, respectively, we have
\begin{align}
\begin{split}\label{e:var4}
\bE\big[\|C_+ - \wh C_+\|_F^2\big]
&= \frac 1m\sum_{i,j=1}^N\big\<F_m(A_0)Qg_{ij},Qg_{ij}\big\> = \sum_{i,j=1}^N\bigg\|\frac 1m\sum_{k=0}^{m-1}K_0^kQg_{ij}\bigg\|^2
\end{split}
\end{align}
and
\begin{align}
\begin{split}\label{e:var5}
\bE\big[\|C - \wh C\|_F^2\big]
&= \frac 1m\sum_{i,j=1}^N\big\<F_m(A_0)Q\psi_{ij},Q\psi_{ij}\big\> = \sum_{i,j=1}^N\bigg\|\frac 1m\sum_{k=0}^{m-1}K_0^kQ\psi_{ij}\bigg\|^2,
\end{split}
\end{align}
where $F_m$ denotes the well known Fej\'er kernel
\[
F_m(t) = \sum_{|k|\le m-1}\big(1-\tfrac{|k|}m\big)e^{ikt} = \frac 1m\left(\frac{1 - \cos(mt)}{1-\cos t}\right)^2.
\]}
\begin{proof}%[Proof of Lemma \ref{l:fejer}]
We set $q_m(z) := zp_m(z)$ and observe that
\begin{align*}
\Re q_m(e^{it})
&= \sum_{k=1}^{m-1}\big(1-\tfrac km\big)e^{ikt} + \sum_{k=1}^{m-1}\big(1-\tfrac km\big)e^{-ikt} = -1 + \sum_{|k|\le m-1}\big(1-\tfrac{|k|}m\big)e^{ikt} = F_m(t) - 1.
\end{align*}
This proves the first equality in \eqref{e:var5} as $\sum_{i,j=1}^N\|Q\psi_{ij}\|^2 = \bE_0$. Let $E$ denote the spectral measure of the unitary operator $K_0$. For $\Delta\in\frakB(\T)$ we compute
\begin{align*}
\<E(\Delta)Qg_{ij},Qg_{ji}^*\>
&= \<E(\Delta)Qg_{ij},\psi_j\cdot K^*\psi_i\> = \<\psi_j\cdot E(\Delta)Qg_{ij},K^*\psi_i\> = \big\<K(\psi_j\cdot E(\Delta)Qg_{ij}),\psi_i\big\>\\
&= \<K\psi_j\cdot KE(\Delta)Qg_{ij},\psi_i\> = \<KE(\Delta)Qg_{ij},\psi_i\cdot K\psi_j\> = \<KE(\Delta)Qg_{ij},Qg_{ij}\>,
\end{align*}
and hence
\begin{align*}
\Re\<p_m(K_0)Qg_{ij},Qg_{ji}^*\>
&= \Re\int p_m(\la)\,d\<E_\la Qg_{ij},Qg_{ji}^*\> = \Re\int p_m(\la)\,d\<K_0E_\la Qg_{ij},Qg_{ij}\>\\
&= \Re\int\la p_m(\la)\,d\<E_\la Qg_{ij},Qg_{ij}\> = \big\<\Re q_m(K_0)Qg_{ij},Qg_{ij}\big\>\\
&= \|Qg_{ij}\|^2 - \<F_m(A_0)Qg_{ij},Qg_{ij}\>,
\end{align*}
which yields the first equality in \eqref{e:var4}. Now, for $f\in L^2_0(\mu)$ we compute
\begin{align*}
\bigg\|\frac 1m\sum_{k=0}^{m-1}K_0^kf\bigg\|^2 = \frac 1{m^2}\sum_{k,\ell=0}^{m-1}\big\<K_0^{k-\ell}f,f\big\> = \frac 1{m}\sum_{k=-(m-1)}^{m-1}\frac{m-|k|}{m}\big\<K_0^kf,f\big\> = \<F_m(A_0)f,f\>.
\end{align*}
The lemma is proved.
\end{proof}

\begin{rem}
The representation $\big\|\frac 1m\sum_{k=0}^{m-1}K_0^kf\big\|^2 = \<F_m(A_0)f,f\>$ of the square norm of the ergodic series by means of the Fej\'er kernel is well known, see \cite{KachPodv18}.
\end{rem}

The main result of this section is the following extended version of Theorem \ref{thm:main_quadratic} which is valid for exponents $\alpha\in (0,2)$. For its formulation, we define
\[
C(\alpha) = 
\begin{cases}
\frac{4-3\alpha}{1-\alpha}&\text{for $\alpha\in (0,1)$}\\
3&\text{for $\alpha = 1$}\\
\frac{3}{(\alpha-1)(2-\alpha)}&\text{for $\alpha\in (1,2)$}.
\end{cases}
\]
and
\[
C(\alpha,\kappa,\theta) := \max\left\{\frac{2}{1-\cos\theta},\kappa C(\alpha)\right\}.
\]
Moreover, recall the constant
\[
M := \frac{8(1+\|C^{-1}\|_F^2\|C_+\|_F^2)^2}{\|C_+\|_F^2}\cdot\max\{\bE_0,\bE_+\}.
\]

\textbf{Theorem~\ref{thm:main_quadratic}.} (Extended version) \textit{ Assume that $K$ is unitary and suppose that there exist $\alpha\in (0,2)$, $\theta\in (0,1/2)$, and $\kappa\ge 0$ such that for each $f\in\calF$,
\begin{align}\label{e:meas2}
\mu_f(S_{\gamma})\le\kappa\cdot\mu_f(S_\theta)\cdot\gamma^\alpha, \quad \gamma\in (0,\theta].
\end{align}
Then for $\veps\in (0,2)$ we have
\begin{align*}
\bP\big(\|C^{-1}C_+ - \wh C^{-1}\wh C_+\|_F > \veps\big)\,\le\,\frac{C(\alpha,\kappa,\theta)M}{m^\alpha\veps^2},
\end{align*}
where
\begin{align}\label{e:const}
C(\alpha,\kappa,\theta) := \max\left\{\frac{2}{1-\cos\theta},\frac{3\kappa}{(\alpha-1)(2-\alpha)}\right\}.
\end{align}
If $\kappa=0$, then
\begin{align*}
\bP\big(\|C^{-1}C_+ - \wh C^{-1}&\wh C_+\|_F > \veps\big)\,\le\,\frac{M}{(1-\cos\theta)\cdot m^2\veps^2}.
\end{align*}
}
\begin{proof}
Theorem \ref{thm:main_quadratic} is proved by combining Lemma \ref{l:prob_absch} with a result on probabilistic bounds on the errors $\|C_+ - \wh C_+\|_F$ and $\|C - \wh C\|_F$. These bounds are provided by the Proposition \ref{p:probs_det_app} below.
\end{proof}

We set $C_0 = C$ and $\wh C_0 = \wh C$.

\begin{prop}\label{p:probs_det_app}
Let the conditions of the extended version of Theorem \ref{thm:main_quadratic} be satisfied. Then for $\veps>0$ we have
\begin{align}\label{e:superlinear}
\bP\big(\|C_i - \wh C_i\|_F > \veps\big)\,\le\,\frac{C(\alpha,\kappa,\theta)\cdot\bE_i}{m^\alpha\veps^2},\qquad i\in\{+,0\}.
\end{align}
If $\mu_f(S_\theta)=0$ for some $\theta\in (0,\pi)$ and all $f\in\calF$, then
\begin{align}\label{e:easier}
\bP\big(\|C_i - \wh C_i\|_F > \veps\big)\,\le\,\frac{2\bE_i}{(1-\cos\theta)m^2\veps^2},\qquad i\in\{+,0\}.
\end{align}
\end{prop}
\begin{proof}
We prove the statements for $i=0$. Similar reasonings apply to the case $i=+$, respectively.

To show the second statement, assume that there exists $\theta\in (0,1/2)$ such that $\mu_f(S_\theta)=0$ for all $f\in\calF$, and let $r(z) := \sum_{k=0}^{m-1}z^k$. Then
\begin{align*}
\bE\big[\|C - \wh C\|_F^2\big]
&= \sum_{i,j=1}^N\bigg\|\frac 1m\sum_{k=0}^{m-1}K_0^kQ\psi_{ij}\bigg\|^2 = \frac 1{m^2}\sum_{i,j=1}^N\|r(K_0)Q\psi_{ij}\|^2 = \frac 1{m^2}\sum_{i,j=1}^N\int_{\T\backslash S_\theta}|r(z)|^2\,d\mu_{Q\psi_{ij}}(z)\\
&= \frac 1{m^2}\sum_{i,j=1}^N\int_{\T\backslash S_\theta}\left|\frac{1-z^m}{1-z}\right|^2\,d\mu_{Q\psi_{ij}}(z)\,\le\,\frac 1{m^2}\sum_{i,j=1}^N\frac{4}{|1-e^{i\theta}|^2}\|Q\psi_{ij}\|^2 = \frac{2\bE_0}{(1-\cos\theta)m^2}.
\end{align*}
Hence, \eqref{e:easier} follows from Markov's inequality.

Let us now assume that \eqref{e:meas2} holds for all $f\in\calF$ and all $\gamma\in (0,\theta]$. Let $f = Q\psi_{ij}$ for some fixed pair $i,j\in [1:N]$. Then $f = f_1 + f_2$, where $f_1 = E(\T\backslash S_\theta)f$ and $f_2 = E(S_\theta)f$. Since $\mu_{f_1}(S_\theta)=0$, we obtain as above that
\[
\bigg\|\frac 1m\sum_{k=0}^{m-1}K_0^kf_1\bigg\|^2\,\le\,\frac{2\|f_1\|^2}{(1-\cos\theta)m^2}.
\]
For $f_2$ and all $\gamma\in (0,1/2]$ we infer from \eqref{e:meas2} that
\begin{align*}
\mu_{f_2}(S_\gamma)
&= \mu_f(S_\gamma\cap S_\theta) = \mu_f(S_{\min\{\gamma,\theta\}})\le\kappa\mu_f(S_\theta)\min\{\gamma,\theta\}^\alpha \le\kappa\mu_f(S_\theta)\gamma^\alpha = \kappa\|E(S_\theta)f\|^2\gamma^\alpha = \kappa\|f_2\|^2\gamma^\alpha.
\end{align*}
Hence, Theorem 2 in \cite{KachPodv18} implies
\[
\bigg\|\frac 1m\sum_{k=0}^{m-1}K_0^kf_2\bigg\|^2\,\le\,\frac{\kappa\|f_2\|^2C(\alpha)}{m^\alpha}.
\]
We conclude
\begin{align*}
\bigg\|\frac 1m\sum_{k=0}^{m-1}K_0^kf\bigg\|^2
&\le \bigg\|\frac 1m\sum_{k=0}^{m-1}K_0^kf_1\bigg\|^2 + \bigg\|\frac 1m\sum_{k=0}^{m-1}K_0^kf_2\bigg\|^2\,\le\,\frac{2\|f_1\|^2}{(1-\cos\theta)m^2} + \frac{\kappa\|f_2\|^2C(\alpha)}{m^\alpha}\\
&\le \max\left\{\frac{2}{1-\cos\theta},\kappa C(\alpha)\right\}\frac{\|f\|^2}{m^\alpha} = C(\alpha,\kappa,\theta)\frac{\|f\|^2}{m^\alpha},
\end{align*}
and therefore
\begin{align*}
\bE\big[\|C - \wh C\|_F^2\big]
&= \sum_{i,j=1}^N\bigg\|\frac 1m\sum_{k=0}^{m-1}K_0^kQ\psi_{ij}\bigg\|^2
\le C(\alpha,\kappa,\theta)\sum_{i,j=1}^N\frac{\|Q\psi_{ij}\|^2}{m^\alpha} = \frac{C(\alpha,\kappa,\theta)\cdot\bE_0}{m^\alpha}.
\end{align*}
The proposition is proved.
\end{proof}

We now turn to the proof of Corollary \ref{c:thin_meas}.

\textbf{Corollary~\ref{c:thin_meas}.}
\textit{Assume that $K_0$ is of the form $K_0 = \sum_{n\in\N}e^{2\pi it_n}\<\,\cdot\,,f_n\>f_n$, where $t_n\in [-1/2,1/2)$, $n\in\N$, and $(f_n)$ is an orthonormal basis of $L_0^2(\mu)$. Then an $f\in L^2_0(\mu)$ satisfies \eqref{e:thin_meas} if $(\<f,f_n\>)_{n\in\N}\in\ell^2_w(\N)$, where $w_n = |t_n|^{-\alpha}$.}
\begin{proof}
Suppose that $(\<f,f_n\>)_{n\in\N}\in\ell^2_w(\N)$, i.e., $S := \sum_{n\in\N}|\<f,f_n\>|^2w_n < \infty$. Hence,
\[
\mu_f(S_\gamma) = \sum_{|t_n|\le\gamma}|\<f,f_n\>|^2\,\le\,\sum_{|t_n|\le\gamma}\frac{\gamma^\alpha}{|t_n|^\alpha}|\<f,f_n\>|^2\,\le\,S\gamma^\alpha.
\]
As $\theta$ can be chosen arbitrarily, \eqref{e:thin_meas} holds with $\kappa = S/\|f\|^2$.
\end{proof}

\subsection{Proof of Theorem \ref{t:main1} in Section~\ref{sec:proofs}}
In this subsection, we let $\mu = \nu$ satisfying \eqref{e:nu}, i.e., we consider the case \textbf{(S2)} of i.i.d.~sampling and assume that the observables $\psi_1,\ldots,\psi_N$ are strongly $\mu$-linearly independent, see Definition \ref{d:lin_ind}. Let us recall the main result on the learning error in case \textbf{(S2)} as stated in Theorem~\ref{t:main1}.

%\red{renumber}
%\begin{thm}\label{t:main1_app}
\textbf{Theorem~\ref{t:main1}.}
\textit{
Assume that $C_+\neq 0$. Let $\veps>0$ and set $\sigma = 2\|C^{-1}\|_F\|C_+\|_F + \veps$. Then
\begin{align*}
\bP\big(\|C^{-1}C_+ &- \wh C^{-1}\wh C_+\|_F > \veps\big)\le\frac{\sigma^2}{m\veps^2}\left[\big(L\|C_+\|_F^{-2} + \|C^{-1}\|_F^2\big)\|\vphi\|^2 - 2\right].
\end{align*}
If, in addition, $\vphi\in L^\infty(\mu)$, then
\begin{align*}
\bP\big(\|C^{-1}C_+ - \wh C^{-1}\wh C_+\|_F > \veps\big)\le 2\exp\left(-\frac{m\veps^2\|C_+\|_F^2}{2\tau^2(1+L)^2}\right) + 2\exp\left(-\frac{m\veps^2}{8\tau^2\|C^{-1}\|_F^2}\right),
\end{align*}
where $\tau = \sigma\|\vphi\|_\infty$.}
%\end{thm}
\begin{proof}
    Theorem \ref{t:main1} is an immediate consequence of the following Proposition~\ref{p:prob_est1} below, which deduces a probabilistic error bound on $C_+ - \wh C_+$ and $C - \wh C$, and Lemma~\ref{l:prob_absch} which allows to straightforwardly infer the claimed bound on $C^{-1}C_+ - \wh C^{-1} \wh C_+$.
\end{proof}

\begin{prop}\label{p:prob_est1}
The following probabilistic bounds on the estimation errors hold:
\begin{align}
\begin{split}\label{e:first_two}
\bP\big(\|C_+ - \wh C_+\|_F > \veps\big)
\le\frac{\big\<K\vphi,\vphi\big\> - \|C_+\|_F^2}{m\veps^2}
\qquad\text{and}\qquad
\bP\big(\|C - \wh C\|_F > \veps\big)
\le\frac{\|\vphi\|^2 - \|C\|_F^2}{m\veps^2}.
\end{split}
\end{align}
If $\vphi\in L^\infty(\mu)$, then
\begin{align}
\begin{split}\label{e:second_two}
\bP\big(\|C_+ - \wh C_+\|_F > \veps\big)
\le 2 \,e^{-\frac{m\veps^2}{2(1+L)^2\|\vphi\|_\infty^2}}
\qquad\text{and}\qquad
\bP\big(\|C - \wh C\|_F > \veps\big)
\le 2 \,e^{-\frac{m\veps^2}{8\|\vphi\|_\infty^2}}.
\end{split}
\end{align}
\end{prop}
\begin{proof}%[Proof of Proposition \ref{p:prob_est1}]
We set $\Phi_k := \Psi(x_k)\Psi(y_k)^\top$. As in the proof of Theorem \ref{t:var}, we compute
\begin{align*}
\bE\big[\|C_+ - \wh C_+\|_F^2\big] = \frac 1m\big[\<K\vphi,\vphi\> - \|C_+\|_F^2\big] + \frac 2m\sum_{k=1}^{m-1}(1 - \tfrac km)\bE\big[\big\<C_+ - \Phi_k,C_+ - \Phi_0\big\>_F\big].
\end{align*}
The sum $\frac 2m\sum_{k=1}^{m-1}\ldots$ in the last expression vanishes as $\Phi_k$ and $\Phi_0$ are stochastically independent. Now, apply Markov's inequality to the non-negative random variable $\|C_+ - \wh C_+\|_F^2$ to obtain the first probabilistic bound. The second is proved similarly.

Now, assume that $\vphi\in L^\infty(\mu)$ and set $S := \|\vphi\|_{L^\infty(\mu)}$. Since $\|K\|\le L$, we observe that
\begin{align*}
\|C_+\|_F^2
&= \sum_{i,j=1}^N|\<\psi_i,K\psi_j\>|^2\le\sum_{i,j=1}^N\|\psi_i\|^2L^2\|\psi_j\|^2 = L^2\|\vphi\|_1^2\le L^2S^2.
\end{align*}
Next, let $(X,Y)\sim\mu_{0}$. Then $\|\Psi(X)\Psi(Y)^\top\|_F^2 = \vphi(X)\vphi(Y)$, see \eqref{e:krank}. We shall prove that $\vphi(X)\vphi(Y)\le S^2$ a.s.. To see this, we note that $X\sim\mu$ and thus $\vphi(X)\le S$ a.s., so that
\begin{align*}
\bP\big(\vphi(X)\vphi(Y)>S^2\big)
&\le\bP(\vphi(Y)>S) = \bP\big((X,Y)\in\calX\times\vphi^{-1}((S,\infty))\big)\\
&= \int\rho(x,\vphi^{-1}((S,\infty)))\,d\mu(x)\le L^2\cdot\mu(\vphi^{-1}((S,\infty))) = 0.
\end{align*}
Thus, we obtain that $\|\Psi(x_k)\Psi(y_k)^\top\|_F\le S$ a.s.\ for all $k=0,\ldots,m-1$.

Consider the random matrices $D_k := C_+ - \Psi(x_k)\Psi(y_k)^\top$, $k=0,\ldots,m-1$. These are stochastically independent, $\bE[D_k]=0$, and $\|D_k\|_F\le (1+L)S$ a.s.. Hence, Hoeffding's inequality for bounded independent random variables in Hilbert spaces (see Corollary A.5.2 in \cite{Moll21} and Theorem 3.5 in \cite{Pine94}) implies that
\[
\bP\big(\|C_+ - \wh C_+\|_F > \veps\big) = \bP\left(\left\|\tfrac 1m\sum_{k=0}^{m-1}D_k\right\|_F > \veps\right)\,\le\,2\cdot e^{-\frac{m\veps^2}{2(1+L)^2S^2}},
\]
as claimed. The second bound is proved similarly.
\end{proof}

\begin{rem}
The bound on $\|C-\wh C\|_F$ in Proposition \ref{p:prob_est1} had been found already in Lemma 3.4.1 of \cite{Moll21} (where $\|\vphi\|_\infty$ is replaced by the slightly worse constant $N\cdot\max_{i=1,\ldots,N}\|\psi_i\|_\infty^2$).
\end{rem}

\section{Further aspects of the Koopman operator}\label{a:just}
The following considers both cases of i.i.d.~and ergodic sampling, that is, we let $\mu\in\{\nu,\pi\}$ as in \textbf{(S1)} and \textbf{(S2)}.
%\textcolor{red}{Please merge B and D}
\subsection{Well-definedness of the Koopman operator}
Note that \eqref{e:inv} and \eqref{e:nu} imply $\tau\ll\mu$, where $\tau(A) := \int\rho(x,A)\,d\mu(x)$, with a density $g\in L^\infty(\mu)$. In particular, $f\in L^p(\mu)$ implies $f\in L^p(\tau)$, since $\int |f|^p\,d\tau = \int |f|^pg\,d\mu\le \|g\|_\infty\int |f|^p\,d\mu$.

Note that for every simple function $f$ on $\calX$ we have
\[
\int f\,d\tau = \int\int f(y)\,\rho(x,dy)\,d\mu(x).
\]
Now, let $f\in L^1(\mu)$, $f\ge 0$. Then $f\in L^1(\tau)$, and there exists a sequence of simple functions $(f_n)_{n\in\N}$, $0\le f_n\le f$, such that $f_n\upto f$ as $n\to\infty$. Hence,
\[
\int f\,d\tau = \int (f-f_n)\,d\tau + \int\int f_n(y)\,\rho(x,dy)\,d\mu(x).
\]
The first integral approaches zero as $n\to\infty$ by dominated convergence ($0\le f-f_n\le 2|f|$), and the second integral tends to $\int\int f(y)\,\rho(x,dy)\,d\mu(x)$ as $n\to\infty$ by Beppo-Levi. In particular, $\int f(y)\,\rho(x,dy) < \infty$ for $\mu$-a.e.\ $x\in\calX$, which is the definition of the Koopman operator.

\subsection{The adjoint of the Koopman operator}
 For $f\in L^1(\mu)$ define the signed measure
\[
\mu_f(A) := \int\rho(x,A)f(x)\,d\mu(x),\qquad A\in\frakB(\calX).
\]
If $\mu(A)=0$, then $\int\rho(x,A)\,d\mu(x)\le L\mu(A) = 0$ (where $L=1$ if $\mu=\pi$), hence $\rho(x,A)=0$ for $\mu$-a.e.\ $x\in\calX$, and thus $\mu_f(A)=0$. Therefore, there exists a unique $Pf\in L^1(\mu)$ such that $\mu_f(A) = \int_A Pf\,d\mu$ holds for all $A\in\frakB(\calX)$. It is now easily seen that $P : L^1(\mu)\to L^1(\mu)$ is linear. Moreover, if $f\in L^1(\mu)$ and $\Delta_+ = \{f\ge 0\}$, $\Delta_- = \{f < 0\}$, we have $\mu_f = \mu_f^+ - \mu_f^-$, where $\mu_f^pm(A) = \pm\int_{\Delta_\pm}\rho(x,A)\,d\mu(x)$, thus
\[
\int|Pf|\,d\mu = |\mu_f|(\calX) = \mu_f^+(\calX) + \mu_f^-(\calX) = \int|f|\,d\mu.
\]
Hence, $P$ is an isometry on $L^1(\mu)$. It is moreover easy to see that $P$ is a Markov operator, i.e., $P\one = \one$ and $Pf\ge 0$ if $f\ge 0$.

For a Borel set $A\in\frakB(\calX)$ and $f\in L^q(\mu)$ (where $\frac 1p + \frac 1q = 1$) we have
\[
\int_A K_p^*f\,d\mu = \<K_p^*f,\one_A\>_{L^q,L^p} = \<f,K_p\one_A\>_{L^q,L^p} = \int\rho(x,A)f(x)\,d\mu(x) = \mu_f(A) = \int_A Pf\,d\mu.
\]
This implies $K_p^*f = Pf$. In particular, $P$ maps $L^q(\mu)$ into $L^q(\mu)$ for all $q\in [1,\infty]$ and coicides there with $K_p^*$ for $q\neq 1$.

\begin{lem}
For $f\in L^2(\mu)$ we have $(Pf)^2\le Pf^2$ $\mu$-a.e.
\end{lem}
\begin{proof}
Let $f$ be a simple function, i.e., $f = \sum_{i=1}^n a_i\one_{A_i}$, where the $A_i$ are mutually disjoint and $\bigcup_{i=1}^nA_i = \calX$. Then $\sum_{i=1}^nP\one_{A_i} = P\one = \one$, hence, by convexity of $z\mapsto z^2$,
\[
(Pf)^2(x) = \Big(\sum_{i=1}^na_i(P\one_{A_i})(x)\Big)^2\,\le\,\sum_{i=1}^na_i^2(P\one_{A_i})(x) = (Pf^2)(x)
\]
and therefore $(Pf)^2\le Pf^2$. Similarly, $|Pf|\le P|f|$. If $f\in L^2(\mu)$, let $(f_n)$ be a sequence of simple functions such that $\|f_n-f\|_{L^2(\mu)}\to 0$ as $n\to\infty$. Then, for every $A\in\frakB(\calX)$,
\begin{align*}
\int_A \big[(Pf)^2 - Pf^2\big]\,d\mu
&\le\int_A ((Pf)^2 - (Pf_n)^2)\,d\mu + \int_A (Pf_n^2 - Pf^2)\,d\mu\\
&\le \int |P(f-f_n)P(f+f_n)|\,d\mu + \|P(f_n^2-f^2)\|_{L^1(\mu)}\\
&\le \|P(f_n-f)\|_{L^2(\mu)}\|P(f_n+f)\|_{L^2(\mu)} + \|f_n^2-f^2\|_{L^1(\mu)}\\
&\le 2\|f_n-f\|_{L^2(\mu)}\|f_n+f\|_{L^2(\mu)}.
\end{align*}
This proves $(Pf)^2\le Pf^2$ $\mu$-a.e.\ for all $f\in L^2(\mu)$.
\end{proof}

\begin{lem}\label{l:geht}
Let $f,g,h\in L^2(\mu)$ such that $g^2Pf^2\in L^1(\mu)$. Then $f\cdot K_1(gh)\in L^1(\mu)$ and
\[
\int f\cdot K_1(gh)\,d\mu = \int [g\cdot K^*f]\cdot h\,d\mu.
\]
\end{lem}
\begin{proof}
Let $(f_n)\subset L^\infty(\mu)$ be a sequence of simple functions such that $f_n\to f$ and $|f_n|\nearrow|f|$ pointwise $\mu$-a.e.\ as $n\to\infty$. Then
\[
\int|K^*(f_n-f)|\,d\mu\,\le\,\int K^*|f_n-f|\,d\mu = \int|f_n-f|\,d\mu,
\]
which converges to zero as $n\to\infty$ by dominated convergence ($|f_n-f|\le 2|f|$). Hence, there exists a subsequence $(f_{n_k})$ such that $K^*f_{n_k}\to K^*f$ $\mu$-a.e.\ as $k\to\infty$. WLOG, we may therefore assume that $K^*f_n\to K^*f$ $\mu$-a.e.\ as $n\to\infty$. By monotone convergence,
\begin{align*}
\int |f||K_1(gh)|\,d\mu
&= \lim_{n\to\infty}\int |f_n||K_1(gh)|\,d\mu\le \limsup_{n\to\infty}\int |f_n|\cdot K_1(|gh|)\,d\mu\\
&= \limsup_{n\to\infty}\int K_1^*|f_n|\cdot |gh|\,d\mu \,\le\, \int |g|P|f|\cdot |h|\,d\mu,
\end{align*}
which is a finite number. Hence, indeed, $f\cdot K_1(gh)\in L^1(\mu)$ and, by dominated convergence,
\begin{align*}
\int f\cdot K_1(gh)\,d\mu
&= \lim_{n\to\infty}\int f_n\cdot K_1(gh)\,d\mu = \lim_{n\to\infty}\int [g\cdot K_1^*f_n]\cdot h\,d\mu = \int [g\cdot K^*f]\cdot h\,d\mu,
\end{align*}
as claimed.
\end{proof}

\section{Auxiliary results}\label{a:proofs}

In this section, we provide auxiliary results on the invertibility of the considered matrices. % and a result on the inverse of random matrices.

\begin{defn}\label{d:lin_ind}
Let $\mu$ be an arbitrary measure on $\calX$, and let measurable functions $\psi_1,\ldots,\psi_N : \calX\to\R$ be given.

{\bf (a)} We say that $\psi_1,\ldots,\psi_N$ are linearly independent w.r.t.\ the measure $\mu$ \braces{or simply $\mu$-linearly independent} if $\sum_{j=1}^N\la_j\psi_j = 0$ $\mu$-a.e.\ implies $\la_1 = \dots = \la_N = 0$.

\smallskip\noindent
{\bf (b)} We say that $\psi_1,\ldots,\Psi_N$ are strongly linearly independent w.r.t.\ the measure $\mu$ \braces{or simply strongly $\mu$-linearly independent} if $\mu\left(\sum_{j=1}^N\la_j\psi_j = 0\right) > 0$ implies $\la_1 = \dots = \la_N = 0$.
\end{defn}

If $\psi_1,\ldots,\Psi_N$ are strongly $\mu$-linearly independent, it follows in particular that the sets of zeros $\psi_j^{-1}(\{0\})$ are null sets (w.r.t.\ $\mu$). Furthermore, note that the following implications hold:
\begin{center}
strong $\mu$-linear independence$\quad\Lra\quad$ $\mu$-linear independence$\quad\Lra\quad$linear independence.
\end{center}

The following lemma holds for both cases {\bf (S1)} and {\bf (S2)}.

\begin{lem}\label{l:mat_repr}
Let $\mu\in\{\nu,\pi\}$ and let $\psi_1,\ldots,\psi_N$ be $\mu$-linearly independent. Then, the matrix $C$ is invertible, and the matrix representation $\wt K_\bV$ of the compression $P_\bV K|_\bV$ of~$K$ to $\bV$ %w.r.t.\ the basis $\calD$ of $\bV$ 
is given by $\wt K_\bV = C^{-1}C_+$.
\end{lem}
\begin{proof}%[Proof of Lemma \ref{l:mat_repr}]
Let $v\in\R^N$ such that $Cv=0$. Then we have $\sum_{j=1}^N\<\psi_i,\psi_j\>v_j = 0$ and hence $\<\psi_i,\sum_{j=1}^N v_j\psi_j\> = 0$ for all $i\in [N]$. But this implies $\sum_{j=1}^N v_j\psi_j = 0$ $\mu$-a.e.\ and thus $v_1=\cdots=v_N=0$ as the $\psi_j$ are $\mu$-linearly independent. Hence, $C$ is indeed invertible.

For any $j\in [N]$, we have $P_\bV K\psi_j = \sum_{i=1}^N a_{ij}\psi_i$ with some $a_{ij}\in\R$ forming the matrix $\wt K_\bV = (a_{ij})_{i,j=1}^N$. Next, for $i,j\in [N]$,
\[
(C_+)_{ij} = \<\psi_i,K\psi_j\> = \<\psi_i,P_\bV K\psi_j\> = \Big\<\psi_i,\sum_{\ell=1}^N a_{\ell j}\psi_\ell\Big\> = \sum_{\ell=1}^N \<\psi_i,\psi_\ell\>a_{\ell j} = (C\wt K_\bV)_{ij},
\]
and the claim follows.
\end{proof}

%\textcolor{red}{The proof of the following lemma can be found in Appendix \ref{a:proofs}.

%\textcolor{red}{This formula raises the immediate question under which conditions the random matrix $\wh C$ is invertible almost surely. The proofs of the next two lemmas and the definition of (strong) $\mu$-linear independence have been moved to Appendix \ref{a:proofs}.
\begin{lem}\label{l:case1}
In case {\bf (S2)}, $\wh C$ is invertible a.s.\ if and only if $m\ge N$ and $\psi_1,\ldots,\psi_N$ are strongly $\mu$-linearly independent.
\end{lem}
\begin{proof}%[Proof of Lemma \ref{l:case1}]
First of all, note that $\wh C$ is invertible if and only if the rank of $\Psi_X$ equals $N$.

Assume that the $\psi_j$ are strongly linearly independent and $m\ge N$. Define the random variables $Y_j := \Psi(x_j)\in\R^N$, $j=0,\ldots,m-1$, and the pushforward measure $\tau := \mu\circ\Psi^{-1}$. Then the $Y_j$ are $\tau$-distributed and independent, and we have $\bP(\la^\top Y_j=0)=0$ for all $j$ and all $\la\in\R^N\backslash\{0\}$.

We shall show that $Y_0,\ldots,Y_{N-1}$ are linearly independent a.s. For this, fix $j\in\{0,\ldots,N-1\}$, let $J\subset\{0,\ldots,N-1\}\backslash\{j\}$ with $1\le k\le N-1$ elements, and define
\[
V := \big\{(y_1,\ldots,y_{k+1})\in (\R^N)^{k+1} : y_{k+1}\in\linspan\{y_1,\ldots,y_k\}\big\}.
\]
Then we have
\begin{align*}
\bP\big(Y_j\in\linspan\{Y_i : i\in J\}\big)
&= \tau^{k+1}(V) = \int_{(\R^N)^{k}}\int_{\linspan\{y_1,\ldots,y_k\}}\,d\tau(y)\,d\tau^{k}(y_1,\ldots,y_k)\\
&= \int_{(\R^N)^{k}}\bP\big(Y_j\in\linspan\{y_1,\ldots,y_k\}\big)\,d\tau^{k}(y_1,\ldots,y_k).
\end{align*}
Fix $k$ vectors $y_1,\ldots,y_k\in\R^N$. Then there exists $\la\in\R^N\backslash\{0\}$ such that $\la^\top [y_1,\ldots,y_k] = 0$, hence 
\[
\bP\big(Y_j\in\linspan\{y_1,\ldots,y_k\}\big)\le\bP(\la^\top Y_j=0) = 0.
\]
This implies $\bP\big(Y_j\in\linspan\{Y_i : i\in J\}\big)=0$, which proves the claim.

Conversely, assume that $\Psi_X$ has rank $N$ a.s., let $\la\in\R^N\backslash\{0\}$ and set $Z := \{\la^\top\Psi=0\}\subset\calX$. Then
% \begin{align*}
% 0
% &= \mu^m(\{(x_1,\ldots,x_m)\,|\,\exists\la\in\R^N\backslash\{0\} : \la^\top\Psi(x_k)=0\;\forall k=0,\ldots,m-1\})\\
% &\ge \mu^m(\{(x_1,\ldots,x_m)\,|\,\hat\la^\top \Psi(x_k)=0\;\forall k=0,\ldots,m-1\})\\
% &= \mu^m(\{(x_1,\ldots,x_m)\,|\,x_k\in Z\;\forall k=0,\ldots,m-1\})\\
% &= \mu^m(Z^m) = [\mu(Z)]^m > 0,
% \end{align*}
% a contradiction.
\begin{align*}
0 = \bP(\la^\top\Psi_X=0) = \bP(x_k\in Z\;\forall k=0,\ldots,m-1) = \mu^m(Z^m) = [\mu(Z)]^m,
\end{align*}
which proves that $\psi_1,\ldots,\psi_N$ are strongly $\mu$-linearly independent.
\end{proof}

%\textcolor{red}{Concerning case {\bf (2)}, we have at least a sufficient condition.
\begin{lem}\label{l:case2}
Let $m\ge N+1$. In case {\bf (S1)}, $\wh C$ is invertible a.s.\ if for each $(N-1)$-dimensional subspace $M\subset\R^N$ and $x\in\Psi^{-1}(M)$ we have $\rho(x,\Psi^{-1}(M))=0$.
\end{lem}

\begin{proof}%[Proof of Lemma \ref{l:case2}]
If $\calL\subset\R^N$ is a subspace with $\dim\calL < N-1$, for any $x\in\calX$ there exists an $(N-1)$-dimensional subspace $M\supset\calL$ with $x\in\Psi^{-1}(M)$, and we obtain $\rho(x,\Psi^{-1}(\calL))\le\rho(x,\Psi^{-1}(M)) = 0$. Hence, the assumption holds for all subspaces $M\subset\R^N$.

In the following, let $A := \Psi^{-1}(\{0\})$. Assume that $\mu(A)>0$. Note that $\mu(A)<1$ by $\mu$-linear independence of the $\psi_j$. We show that
\begin{align}
\bP([\Psi(X_1),\ldots,\Psi(X_{N})]\text{ invertible}\,|\,X_0\in A) &= 1\label{e:notinA}\\
\bP([\Psi(X_0),\ldots,\Psi(X_{N-1})]\text{ invertible}\,|\,X_0\notin A) &= 1.\label{e:inA}
\end{align}
Then the claim follows. By assumption, we have $\rho(x,A)=0$ for $x\in A$, hence
\[
\bP(\Psi(X_1)=0\,|\,X_0\in A) = \frac 1{\mu(A)}\int_A\rho(x,A)\,d\mu(x) = 0
\]
%Next, assume that $\bP(\Psi(X_k)\in\linspan\{\Psi(X_0),\ldots,\Psi(X_{k-1})\}\,|\,X_0\in A) = 0$.
Next, let $k\in\{1,\ldots,N-1\}$ and set $M(x_1,\ldots,x_k) := \linspan\{\Psi(x_1),\ldots,\Psi(x_k)\}$ for $x_1,\ldots,x_k\in\calX$. Then
\begin{align*}
\bP\big(\Psi(X_{k+1})&\in M(X_1,\ldots,X_k)\,|\,X_0\in A\big)\\
&= \frac 1{\mu(A)}\int_A\int_{x_1}\dots\int_{x_k}\int_{\Psi^{-1}(M(x_1,\ldots,x_k))}\rho(x_k,dx_{k+1})\rho(x_{k-1},dx_k)\dots\rho(x_0,dx_1)\,d\mu(x_0)\\
&= \frac 1{\mu(A)}\int_A\int_{x_1}\dots\int_{x_k}\rho(x_k,\Psi^{-1}(M(x_1,\ldots,x_k)))\rho(x_{k-1},dx_k)\dots\rho(x_0,dx_1)\,d\mu(x_0).
\end{align*}
Now, note that for all $x_k\in\calX$ we have $x_k\in\Psi^{-1}(M(x_1,\ldots,x_k))$, hence $\rho(x_k,\Psi^{-1}(M(x_1,\ldots,x_k)))=0$. Therefore, we have $\bP\big(\Psi(X_{k+1})\in M(X_1,\ldots,X_k)\,|\,X_0\in A\big) = 0$ for all $k=1,\ldots,N-1$, which shows \eqref{e:notinA}. The relation \eqref{e:inA} can be shown similarly with all indices dropped by one and $A$ replaced by $\calX\backslash A$. The proof for \eqref{e:inA} carries over to the case $\mu(A)=0$ without conditioning and with $\calX\setminus A$ replaced by $\calX$.
\end{proof}

The following result provides an improved version of Theorem 12 in \cite{NuskPeit23}.
\begin{lem}\label{l:prob_absch}
Let $C,D\in\R^{N\times N}$ be such that $C$ is invertible and $D\neq 0$. Let $\wh C,\wh D\in\R^{N\times N}$ be random matrices such that $\wh C$ is invertible a.s.. Then for any sub-multiplicative matrix norm $\|\,\cdot\,\|$ on $\R^{N\times N}$ and any $\veps > 0$ we have
\begin{align*}
\bP\big(\|C^{-1}D - \wh C^{-1}\wh D\| > \veps\big)
\le \bP\Big(\|D - \wh D\| > \tfrac{\veps}{\tau}\|D\|\Big) + \bP\Big(\|C-\wh C\| > \tfrac{\veps}{\tau}\|C^{-1}\|^{-1}\Big),
\end{align*}
where $\tau = 2\|C^{-1}\|\|D\| + \veps$.
\end{lem}
\begin{proof}
We have $C^{-1}D - \wh C^{-1}\wh D = \wh C^{-1}(D - \wh D) + (C^{-1} - \wh C^{-1})D$ and thus
\begin{align*}
\bP\big(\|C^{-1}D - \wh C^{-1}\wh D\| > \veps\big)
&\le \bP\big(\|\wh C^{-1}\|\|D - \wh D\| + \|D\|\|C^{-1}-\wh C^{-1}\| > \veps\big)\\
&\le \bP\big(\|\wh C^{-1}\|\|D - \wh D\| > \tfrac\veps 2\;\vee\; \|D\|\|C^{-1}-\wh C^{-1}\| > \tfrac\veps 2\big)\\
&= \bP\Big(\|D - \wh D\| > \tfrac{\veps/2}{\|\wh C^{-1}\|}\;\vee\; \|C^{-1}-\wh C^{-1}\| > \tfrac{\veps/2}{\|D\|}\Big)\\
&\le\bP\Big(\|D - \wh D\| > \tfrac{\veps/2}{\|C^{-1}\| + \tfrac{\veps/2}{\|D\|}}\;\vee\; \|C^{-1}-\wh C^{-1}\| > \tfrac{\veps/2}{\|D\|}\Big)\\
&\le\bP\Big(\|D - \wh D\| > \tfrac{\veps/2}{\tau}\|D\|\Big) + \bP\Big(\|C^{-1}-\wh C^{-1}\| > \tfrac{\veps/2}{\|D\|}\Big).
\end{align*}
Next, we estimate
\begin{align*}
\|C^{-1} - \wh C^{-1}\| = \|C^{-1}(\wh C - C)\wh C^{-1}\|\,\le\,\|C^{-1}\|\|\wh C - C\|\big(\|\wh C^{-1} - C^{-1}\|+\|C^{-1}\|\big).
\end{align*}
Hence, if $\|C - \wh C\| < \frac{1}{\|C^{-1}\|}$, then
\[
\|C^{-1} - \wh C^{-1}\|\,\le\,\frac{\|C^{-1}\|^2\|C - \wh C\|}{1 - \|C^{-1}\|\|C-\wh C\|}
\]
%In this case, note furthermore that $\frac{\|C^{-1}\|^2\|C - \wh C\|}{1 - \|C^{-1}\|\|C- \wh C\|} > a$, $a>0$, is equivalent to $\|C-\wh C\| > \frac{a}{\|C^{-1}\|(\|C^{-1}\| + a)}$. 
Therefore,
\begin{align*}
\bP\Big(\|C^{-1}-\wh C^{-1}\| > \tfrac{\veps/2}{\|D\|}\Big)
&\le \bP\Big(\|C-\wh C\|\ge\tfrac{1}{\|C^{-1}\|}\;\vee\;\tfrac{\|C^{-1}\|^2\|C - \wh C\|}{1 - \|C^{-1}\|\|C-\wh C\|} > \tfrac{\veps/2}{\|D\|}\Big)\\
&= \bP\Big(\|C-\wh C\|\ge\tfrac{1}{\|C^{-1}\|}\;\vee\;\|C-\wh C\| > \tfrac{\veps/2}{\|C^{-1}\|(\|D\|\|C^{-1}\|+\veps/2)}\Big)\\
&= \bP\Big(\|C-\wh C\| > \tfrac{\veps/2}{\|C^{-1}\|(\|D\|\|C^{-1}\|+\veps/2)}\Big),
\end{align*}
and the lemma is proved.
\end{proof}

\section{Spectral measures of unitary operators}\label{s:unitary}
Let $U$ be a unitary operator in a Hilbert space $\calH$. By the spectral theorem for normal operators in Hilbert spaces (see, e.g., \cite{Conw10}), there exists an operator-valued measure $E$ on the Borel sigma algebra $\calB$ of $\T = \{z\in\C : |z|=1\}$, which has the following properties:
\begin{itemize}
\item $E(\Delta)$ is an orthogonal projection for all $\Delta\in\frakB$.
\item $E(\T\backslash\sigma(U)) = 0$ and $E(\sigma(U)) = I$.
\item $E(\Delta_1\cap\Delta_2) = E(\Delta_1)E(\Delta_2) = E(\Delta_2)E(\Delta_1)$ for $\Delta_1,\Delta_2\in\calB$.
\item $E(\Delta)\calH$ is $U$-invariant for each $\Delta\in\calB$.
\item $\sigma(U|_{E(\Delta)\calH})\subset\sigma(U)\cap\Delta$ for closed $\Delta\in\calB$.
% \item For each $f\in\calH$,
% $$
% \<Uf,f\> = \int_\T z\,d\mu_f(z)\quad\text{and}\quad \|Uf\|^2 = \int_\T |z|^2\,d\mu_f(z),
% $$
% where $\mu_f$ is the measure $\mu_f(\Delta) := \|E(\Delta)f\|^2$, $\Delta\in\calB$.
\end{itemize}
The measure $E$ is called the {\em spectral measure} of $U$. In the finite-dimensional case (i.e., when $U\in\C^{n\times n}$ is a unitary matrix), the projection $E(\Delta)$ is the orthogonal projection onto the sum of eigenspaces corresponding to the eigenvalues of $U$ in $\Delta$. This is also true in the infinite-dimensional case if $U$ has only discrete spectrum in $\Delta$ (i.e., isolated eigenvalues).

Let $g : \T\to\C$ be a bounded measurable function. Then $g(U) := \int g\,dE$ defines a bounded normal operator, and for $f\in\calH$ we have
$$
\<g(U)f,f\> = \int_\T g(z)\,d\mu_f(z)\qquad\text{and}\qquad \|g(U)f\|^2 = \int_\T |g(z)|^2\,d\mu_f(z),
$$
where $\mu_f$ is the measure $\mu_f(\Delta) := \|E(\Delta)f\|^2$, $\Delta\in\calB$.
\end{document}